\documentclass[dvips,10pt]{article}
\title{\textsc{On finite generation and infinite convergence\\ of generalized closures from the theory\\ of cutting planes}}
\author{Gennadiy Averkov\footnote{University of Magdeburg, Universit\"atsplatz 2,  39106 Magdeburg, Germany; email: averkov@math.uni-magdeburg.de}}
\usepackage{amsthm,amsmath,amssymb,enumerate,graphicx,graphpap,empheq}
\usepackage{srcltx}

\newcommand{\rmcmd}[1]{\mathop{\mathrm{#1}}\nolimits}
\newcommand{\conv}{\rmcmd{conv}}
\newcommand{\cone}{\rmcmd{cone}}
\newcommand{\aff}{\rmcmd{aff}}
\newcommand{\lin}{\rmcmd{lin}}
\newcommand{\relintr}{\rmcmd{relint}}
\newcommand{\bd}{\rmcmd{bd}}

\newcommand{\real}{\mathbb{R}}
\newcommand{\natur}{\mathbb{N}}
\newcommand{\integer}{\mathbb{Z}}
\newcommand{\rational}{\mathbb{Q}}
\newcommand{\Aff}{\rmcmd{Aff}}

\newcommand{\term}[1]{\emph{#1}}

\newcommand{\cP}{\mathcal{P}}

\newcommand{\setcond}[2]{\left\{#1 \, : \, #2 \right\}}

\newcommand{\sprod}[2]{\left< #1 \, , \, #2 \right>}
\newcommand{\dotvar}{\, \cdot \,}
\newcommand{\intr}{\rmcmd{int}}

\newcommand{\bM}{\mathbb{M}}

\newcommand{\vol}{\rmcmd{vol}}

\newcommand{\header}[1]{\textup{(#1)}}

\newcommand{\maxfw}{\rmcmd{maxfw}}

\newcommand{\cl}{\rmcmd{cl}}
\newcommand{\cX}{\mathcal{X}}
\newcommand{\ext}{\rmcmd{ext}}
\newcommand{\rec}{\rmcmd{rec}}
\newcommand{\extr}{\rmcmd{extr}}

\newcommand{\red}{R}
\newcommand{\lineal}{\rmcmd{lineal}}
\newcommand{\Rbar}{\overline{R}}

\newcommand{\relint}{\relintr}

\newcommand{\floor}[1]{\left\lfloor#1\right\rfloor}
\newcommand{\eop}{\hspace*{\fill}~$\square$}
\newcommand{\Ch}{\rmcmd{Ch}}
\newcommand{\Sp}{\rmcmd{Sp}}
\newcommand{\Rem}{\rmcmd{Rem}}

\newenvironment{figtabular}[1]{
	\begin{figure}[htb]
		\begin{center}
			\begin{tabular}{#1}
}{
			\end{tabular}

		\end{center}

	\end{figure}
}

\newtheorem{nn}{}[section]
\newtheorem{theorem}[nn]{Theorem}

\newtheorem{proposition}[nn]{Proposition}
\newtheorem{lemma}[nn]{Lemma}
\newtheorem{corollary}[nn]{Corollary}

\theoremstyle{definition}
\newtheorem{remark}[nn]{Remark}
\newtheorem*{acknowledgements*}{Acknowledgements}

\numberwithin{equation}{section}

\begin{document}

\maketitle

\newcommand{\cL}{\mathcal{L}}

\begin{abstract}
	For convex sets $K$ and $L$ in $\real^d$ we define $R_L(K)$ to be the convex hull of all points belonging to $K$ but not to the interior of $L$. Cutting-plane  methods from integer and mixed-integer optimization can be expressed in geometric terms using functionals $R_L$ with appropriately chosen sets $L$. We describe the geometric properties of $R_L(K)$ and characterize those $L$ for which $R_L$ maps polyhedra to polyhedra. For certain natural classes $\cL$ of convex sets in $\real^d$ we consider the functional $R_\cL$ given by $R_\cL(K):= \bigcap_{L \in \cL}R_L(K)$. The functional $R_\cL$ can be used to define various types of closure operations considered in the theory of cutting planes (such as the Chv\'atal closure, the split closure as well as  generalized split closures recently introduced by Andersen, Louveaux and Weismantel). We study conditions on $\cL$ under which $R_\cL$ maps rational polyhedra to rational polyhedra. We also describe the limit of the sequence of sets obtained by iterative application of $R_\cL$ to $K$. A part of the presented material gives generalized formulations and unified proofs of several recent results obtained by various authors.
\end{abstract}

\newtheoremstyle{itsemicolon}{}{}{\mdseries\rmfamily}{}{\itshape}{:}{ }{}
\newtheoremstyle{itdot}{}{}{\mdseries\rmfamily}{}{\itshape}{:}{ }{}
\theoremstyle{itdot}
\newtheorem*{msc*}{2010 Mathematics Subject Classification} 

\begin{msc*}
	Primary 90C11;  Secondary 52A20, 52B20, 90C10.
\end{msc*}


\newtheorem*{keywords*}{Keywords}

\begin{keywords*}
	Chv\'atal closure; cutting plane; disjunctive programming; Gomory cut; max-facet-width; mixed-integer optimization; recession cone; split closure; width.
\end{keywords*}

\section{Introduction}

Let $d \in \natur$. By $\le$ we denote the standard partial order on $\real^d$, that is, for $x,y \in \real^d$ one has $x \le y$ if and only if for each $i \in \{1,\ldots,d\}$ the $i$-th component of $x$ is not larger than the $i$-th component of $y$. The notations $\conv$ and $\intr$ stand for the convex hull and interior, respectively. In what follows $L$ stands for an arbitrary $d$-dimensional closed convex set in $\real^d$, $\cL$ for a nonempty class of $d$-dimensional closed convex sets in $\real^d$ and $K$ for an arbitrary closed convex set in $\real^d$ (not necessarily $d$-dimensional). In this manuscript we study the functionals $R_L$, $R_\cL$ and $R_\cL^i$ (with $i = 0,1,2,\ldots$) defined by
\begin{align*}
	R_L(K) & := \conv\bigl(K \setminus \intr(L)\bigr), \\ 
	R_{\cL}(K) &:= \bigcap_{L \in \cL} R_L(K), \\
	R_\cL^i(K) & := 
	\begin{cases}
		K & \text{if} \ i=0, \\
		R_\cL\bigl(R_\cL^{i-1}(K)\bigr) & \text{if} \ i \in \natur.
	\end{cases}
\end{align*}
We call $R_L(K)$ the \term{$L$-reduction}, $R_\cL(K)$ the \term{$\cL$-closure} and $R_\cL^i(K)$ the $i$-th \term{$\cL$-closure}  of $K$. The term closure in the given context goes back to Chv\'atal \cite{MR0313080}. We remark that $R_\cL$ can be used to define the well-known \term{Chv\'atal closure} and \term{split closure} (see \cite{MR0313080}, \cite{MR1059391}, \cite[Chapter~23]{MR874114}). Furthermore, $R_L$ and $R_\cL$ provide a natural link to \term{disjunctive programming} (see also \cite{MR0376142,MR0452674,MR791175,MR1663099,MR2216796,MR2546330}). If $K$ is a polytope, then $R_L(K)$ is also a polytope, which can be obtained from $K$ by `cutting off' parts of $K$ with hyperplanes which are determined by $L$. See also Fig.~\ref{R_L:illustration} for an illustration. Let us discuss the relation of the introduced functionals to the cutting-plane theory. We shall not discuss computational aspects but rather geometric ideas of cutting-plane methods. We call a subset $\bM$ of $\real^d$ a \term{mixed-integer space} if  
\begin{equation} \label{MI:space}
	\bM=\integer^m \times \real^n
\end{equation}
for integers $m \ge 1$ and $n \ge 0$ satisfying $d=m+n$. A \term{mixed-integer linear problem} is an optimization problem having the following form:
\begin{equation} \label{optim:problem}
	\text{Find $x \in \bM$ maximizing $u^\top x$, subject to $A x \le b$,}
\end{equation}
where $u \in \rational^d$, $A$ is a rational matrix and $b$ is a rational vector. The sizes $A$ and $b$ are assumed to be chosen properly so that the expression $A x \le b$ makes sense. In geometric terms, we are given a rational polyhedron 
\begin{equation} \label{polyh:of:system}
P:= \setcond{x \in \real^d}{A x \le b}
\end{equation}
and we maximize a linear function on $P \cap \bM$. In the case $P = \conv (P \cap \bM)$ all vertices of $P$ belong to $\bM$. Thus, in this case problem \eqref{optim:problem} can be essentially reduced to a problem of linear programming by relaxing the condition $x \in \bM$ to $x \in \real^d$. In the case $P \ne \conv (P \cap \bM)$ there exist closed halfspaces $H^+$ such such that $P \cap \bM \subseteq  H^+$ and $P \not\subseteq H^+$. We call a halfspace $H^+$ as above a \term{cut} for $P$ and we call the boundary of $H^+$ a \term{cutting plane}. If we replace $P$ by $P \cap H^+$ we pass to an equivalent mixed-integer linear problem (in order to preserve the rationality of $P$, we need to assume that $H^+$ is a rational polyhedron). In analytic terms, the system of linear inequalities $A x\le b$ is modified by adding a constraint which defines the cut $H^+$. The above reduction step for $P$ helps `get closer' to the case $P = \conv(P \cap \bM)$ since this step decreases the difference between $P$ and $\conv(P \cap \bM)$. The reduction of $P$ to $P \cap H^+$ is the basic computational step of the \term{cutting-plane methods} (see also \cite{MR1699321,MR1922341,MR2176841,BertWeiBook} for further information). It turns out that, for deriving cuts, functionals $R_L$ can be used. We call a $d$-dimensional closed convex set $L$ in $\real^d$ an \term{$\bM$-free set} if $\intr(L) \cap \bM = \emptyset$.  For an $\bM$-free set $L$ one obviously has $R_L(P) \cap \bM = P \cap \bM$. Thus, a halfspace $H^+$ satisfying $R_L(P) \subseteq H^+$ and $P \not\subseteq H^+$ is a cut for $P$ (see Fig.~\ref{fig-closure-op}). An $\bM$-free set $L$ in $\real^d$ is said to be a \term{maximal $\bM$-free set} if there exists no $\bM$-free set $L'$ with $L \varsubsetneq L'$. Among all cuts determined by $\bM$-free sets the strongest ones arise from maximal $\bM$-free sets. Thus, maximal $\bM$-free sets are of particular importance for the cutting-plane theory (see \cite{arXiv:1003.4365,arXiv:1010.1077,MR1114315} for related results and Fig.~\ref{ex-M-free} for an illustration).

\begin{figtabular}{c}
\unitlength=1.1mm
\begin{picture}(100,35)
	\put(0,-3){\includegraphics[width=100\unitlength]{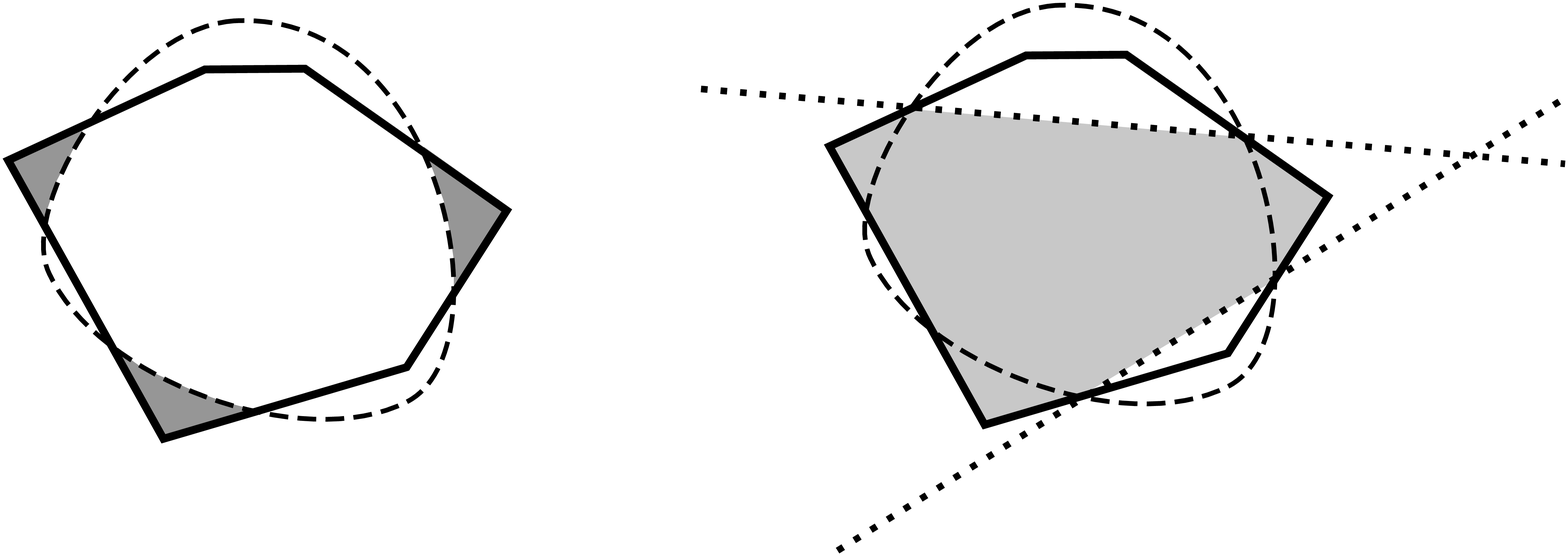}}
	\put(24,28){$L$}
	\put(-3,19){$K$}
	\put(65,17){$R_L(K)$}
\end{picture}
\\
\parbox[t]{0.90\textwidth}{\caption{\label{R_L:illustration}Illustration to the definition of $R_L(K)$ in the case $K$ is a polytope. The thick solid line is the boundary of $K$. The dashed line is the boundary of $L$. The set $K \setminus \intr(L)$ (on the left) is shaded dark. The set $R_L(K)$ (on the right) is shaded light. The dotted lines (on the right) determine how $R_L(K)$ can be constructed by `cutting off' parts of $P$.}}
\end{figtabular}

\begin{figtabular}{c}

\unitlength=1mm
\begin{picture}(90,35)
	\put(-3,0){\includegraphics[width=30\unitlength]{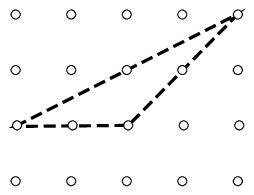}}
	\put(10,10){$L$}
	\put(55,0){\includegraphics[width=30\unitlength]{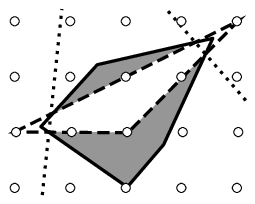}}
	\put(67,5){$P$}
\end{picture}

\\

\parbox[t]{0.90\textwidth}{\caption{\label{fig-closure-op} Generation of cutting planes using $R_L(P)$ in the case $\bM=\integer^2$. The dashed line is the boundary of $L$. The set $P \setminus \intr(L)$ is shaded dark. The dotted lines (on the right) are the two best possible cutting planes which can be generated using $R_L(P)$.}}
\end{figtabular}

\begin{figtabular}{c}
\unitlength=0.9mm
\begin{picture}(140,20)
	\put(0,-2){\includegraphics[width=30\unitlength]{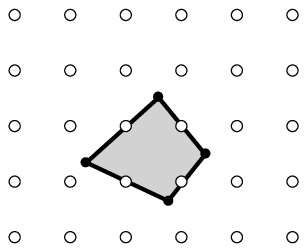}}
	\put(7,1){$P_1$}
	\put(55,-2){\includegraphics[width=30\unitlength]{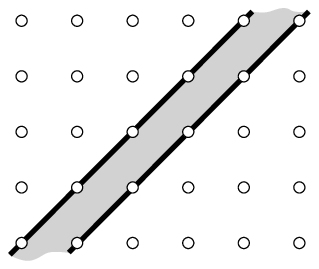}}
	\put(57,7){$P_2$}
	\put(110,-2){\includegraphics[width=30\unitlength]{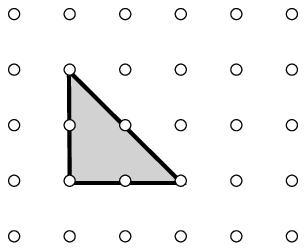}}
	\put(112,1){$P_3$}
\end{picture}

\\

\parbox[t]{0.90\textwidth}{\caption{\label{ex-M-free} Examples of maximal $\bM$-free sets in the case $\bM=\integer^2$.}}
\end{figtabular}

The original cutting-plane method is due to Gomory \cite{MR0102437,MR0174390} (see also \cite[\S23.8]{MR874114}, \cite[Chapter\,14]{MR1637890}). Gomory's method is applied to the problem \eqref{optim:problem} in the case $\bM=\integer^d$ and can be formulated in terms of the functionals $R_L$ such that $L$ is an $\bM$-free set which is an infinite `slab' (that is, the boundary of $L$ consists of two parallel hyperplanes). Balas \cite{MR0290793} introduced cutting-plane methods with respect to general $\bM$-free sets $L$.  For recent results related to cutting-plane methods based on general $\bM$-free sets we refer to \cite{MR2480507,MR2481733,JoergThesis2008,MR2555335,MR2676765,MR2593417,DelPiaWeismantel10,BasuCornuejolsMargot10}. The $\cL$-closure operation $R_\cL$ can be used to estimate the `cutting quality' of methods which generate cuts from sets $L \in \cL$. Furthermore, finite resp. infinite convergence properties of the sequences $\bigl(R_\cL^i(K) \bigr)_{i=0}^{+\infty}$ are related to finite resp. infinite convergence of cutting-plane methods based on $\cL$. With a view toward algorithmic applications, it is natural to ask for a `convenient' description of $R_L$ resp. $R_\cL$ and for conditions on $L$ resp. $\cL$ under which  such a description can be given by finite data. In particular, it is important to know whether for a given $\cL$ the functional $R_\cL$ maps polyhedra to polyhedra (or rational polyhedra to rational polyhedra). The above questions about the properties of $R_L$ and $R_\cL$ were addressed by Andersen, Louveaux and Weismantel \cite{MR2676765}. The convergence of sequences of the form $\bigl(R_\cL^i(K)\bigr)_{i=0}^{+\infty}$ was studied by Owen and Mehrotra \cite{MR1814548} and Del Pia and Weismantel \cite{DelPiaWeismantel10}. In this manuscript we generalize a part of results given in \cite{MR2676765} and \cite{MR1814548}. Furthermore, we also present new results which serve as a natural supplement. In contrast to \cite{MR1814548,MR2676765,DelPiaWeismantel10} we do not restrict considerations to the classes $\cL$ consisting of $\bM$-free sets only. In most of the cases our assumptions on $\cL$ do not involve any mixed-integer space $\bM$. Such a more general setting might be of interest for cutting-plane theory, since it seems possible that cutting-plane methods based on sets which are not necessarily $\bM$-free can also be introduced. In the purely integer case $\bM=\integer^d$, sets $L$ with a fixed positive number of interior integer points could be a natural choice (see also results from \cite{MR651251,MR688412,MR1138580,MR1996360,arXiv:1103.0629} on the geometry of such sets). For example,  one can consider a set $L$ with precisely one interior integer point $z$. With this choice, a cut generated by $L$ removes at most one point of $P \cap \bM$, namely the point $z$ (we recall that $P$ is the polyhedron defined by \eqref{polyh:of:system}). A possible design template would be that a cutting-plane method based on sets $L$ as described above keeps the track of the `best' point removed, that is, the point $z$ with maximal $u^\top z$ among all points which were removed from $P \cap \bM$ during the execution of the method. The aim of the method would be to change $P$ iteratively finally arriving at the situation $P=\emptyset$ or at the situation where $x^\ast \in P$ yielding $\max_{x \in P} u^\top x$ and belonging to $\bM$ can be found. Having riched such a situation, the method can easiliy determine the optimal solution (which is either the best point that was removed or the point $x^\ast$).

Let us give an overview of the main results of the manuscript. In Theorem~\ref{descr:R} we study the properties of $R_L(K)$. In particular, we describe the set of extreme points and the recession cone of $R_L(K)$. In Theorem~\ref{polyhedrality:characterization} we characterize those $L$ for which $R_L$ maps polyhedra to polyhedra. In Theorem~\ref{polyhedrality:of:closures} we consider $\cL$ consisting of rational polyhedra and present a condition on $\cL$ under which $R_\cL$ maps rational polyhedra to rational polyhedra. Even more generally, Theorem~\ref{polyhedrality:of:closures} asserts that, under certain assumptions on $\cL$, the set $R_\cL(P)$ can be `finitely generated', that is, $R_\cL(P) = R_{\cL'}(P)$ for some finite subclass $\cL'$ of $\cL$. The study of $R_\cL$ in the case of $\cL$ consisting of maximal $\bM$-free sets is of particular importance. Therefore, in Theorem~\ref{max-fac-width-of-max-lat-free} we give a simple formulation of the condition on $\cL$ appearing in Theorem~\ref{polyhedrality:of:closures} for this particular case. Theorems~\ref{descr:R} and \ref{polyhedrality:of:closures} generalize the main results of \cite{MR2676765}, while Theorems~\ref{polyhedrality:characterization} and \ref{max-fac-width-of-max-lat-free} are (to the best of our knowledge) new. In Theorem~\ref{convergence:thm} we describe the limit of $R_\cL^i(K)$, as $i \rightarrow +\infty$, under some weak assumptions on $K$ and $\cL$. Theorem~\ref{convergence:thm} is related to Theorem~2 from \cite{DelPiaWeismantel10}. Corollary~\ref{convergence:bounded:case}, which follows from Theorem~\ref{convergence:thm}, is a convergence result extending Theorem~3 from \cite{MR1814548}.

Our proofs use standard tools of affine convex geometry (facial structure of convex sets, recession cones and duality). In the proof of Theorem~\ref{polyhedrality:of:closures} we use the Gordan-Dickson lemma.  The manuscript is organized as follows. Section~\ref{prelim} gives necessary preliminary information, Section~\ref{sect:results} contains the formulations of the main results and  Section~\ref{sect:proofs} presents the proofs.

\section{Preliminaries} \label{prelim}

\subsection{Convex sets}

For information on convex geometry we refer to \cite{MR1216521,MR0274683,MR2335496}.  The elements of $\real^d$ are defined to be columns of $d$ real numbers. The origin of $\real^d$ is denoted by $o$. By $\sprod{\dotvar}{\dotvar}$ we denote the standard scalar product in $\real^d$. We use the functionals $\aff$ (affine hull), $\cone$ (conical hull), $\conv$ (convex hull), $\cl$ (closure), $\intr$ (interior), $\lin$ (linear hull), $\vol$ ($d$-dimensional volume). For $a, b \in \real^d$, we define $[a,b]:=\setcond{(1-t) a + t b}{ 0 \le t \le 1}$. If $a \ne b$, then $[a,b]$ is called a \term{segment} in $\real^d$, and $a$ and $b$ are called the \term{endpoints} of $[a,b]$.

 We use $\relintr(X)$ to denote the \term{relative interior} of $X \subseteq \real^d$, i.e., the interior of $X$ with respect to the Euclidean topology of the affine space $\aff(X)$. If $C$ is a convex set in $\real^d$, then the dimension $\dim (C) $ of $C$ is defined to be the dimension of $\aff(C)$.  For $X, Y \subseteq \real^d$ and $t \in \real$ we introduce
\begin{align*}
	X+Y := & \setcond{x+y}{x \in X, \ y \in Y} & &\text{\term{(Minkowski sum of $X$ and $Y$)},} \\
	X-Y := & \setcond{x-y}{x \in X, \ y \in Y} & &\text{\term{(Minkowski difference of $X$ and $Y$)}}, \\
   t X := & \setcond{t x }{x \in X} & & \text{\term{(scaling of $X$ by factor $t$)}}, \\
	-X := & \setcond{-x}{x \in X} & & \text{\term{(reflection of $X$ in the origin)}.}
\end{align*}

For $a \in \real^d$ we also use the notations $X+a:=X+\{a\}$ and $X-a:=X-\{a\}$. Let $K$ be a nonempty closed convex set in $\real^d$. With $K$ we associate the following functions and sets.
\begin{align}
	h(K,u)  := & \sup_{x \in K} \sprod{u}{x} & & \text{\emph{(support function)}}, \\
	w(K,u)  := & \sup_{x \in K} \sprod{u}{x} - \inf_{ x \in K} \sprod{u}{x} & & \text{\emph{(width function),}} \label{width:function:def} \\
\|u\|_K  := & \inf \setcond{t \ge 0}{u \in t K} & & \text{\emph{(gauge function),}} \\
	K^\circ := & \setcond{y \in \real^d}{\sprod{y}{x} \le 1 \quad \forall \ x \in K} & & \text{\emph{(polar set),}} \\
	\label{rec:K:def}
 	\rec(K) :=& \setcond{y \in \real^d}{x+ t y \in K \quad \forall \ x \in K \ \forall \ t \ge 0 } & & \text{\emph{(recession cone),}} \\
 	\label{lineal:def}
 	\lineal(K) := & \setcond{y \in \real^d}{x+ t y \in K \quad \forall \ x \in K \ \forall \ t \in \real } & & \text{\emph{(lineality space).}}
\end{align}
Above $u \in \real^d$. The gauge function $\|u\|_K$ and the polar set $K^\circ$ are introduced under the assumption $o \in K$. The values of the functions $h(K,u), w(K,u), \|u\|_K$ lie in $\real \cup \{+\infty\}$. It is known that the definitions \eqref{rec:K:def} and \eqref{lineal:def} remain unchanged if one replaces the existential quantifier over $x$ by the universal quantifier. One has
\begin{align}
 \|u\|_K & = h(K^\circ,u), \label{gauge:sup:f:eq}\\
 (K^\circ)^\circ & = K, \label{double:polar:eq} \\
 w(K,u) & = h(K,u) + h(K,-u) = h(K-K,u) \label{width:func:equations},
\end{align}
where \eqref{gauge:sup:f:eq} and \eqref{double:polar:eq} are stated under the assumption $o \in K$. We shall use the following simple proposition. 
\begin{proposition} \label{ray-prop}
	Let $K$ be a closed convex set in $\real^d$ with $o \in \intr(K)$. Let $u \in \real^d \setminus \{o\}$. Consider the ray $I:=\setcond{t u}{t \ge 0}$. Then $\|u\|_K < + \infty$ and, furthermore, the following statements hold.
	\begin{enumerate}[I.]
		\item One has $\|u\|_K = 0$ if and only if the ray $I$ is contained in $K$.
		\item If $\|u\|_K > 0$, then $u / \|u\|_K$ is the unique point of the intersection of $\bd(K)$ and $I$.
	\end{enumerate}
\end{proposition}

The following result can be found in \cite[Theorem~1.1.13]{MR1216521} and \cite[Theorem~6.9]{MR0274683}.

\begin{theorem} \label{char:relint}
	Let $n \in \natur$, $p_1,\ldots,p_n \in \real^d$ and $p \in \real^d$. Then $p \in \relintr \bigl( \conv(\{p_1,\ldots,p_n \})\bigr)$ if and only if there exist $\lambda_1,\ldots,\lambda_n > 0$ with $\lambda_1+ \cdots + \lambda_n =1$ and $p= \lambda_1 p_1 + \cdots + \lambda_n p_n$.
\end{theorem}

The following is a version of Carath\'eodory's theorem. 
\begin{theorem} \header{Carath\'eodory's theorem}. \label{strong:caratheodory}
	Let $X \subseteq \real^d$ and let $x \in \conv(X)$. Then there exists an affinely independent set $Y \subseteq X$ such that $x \in \relintr \bigl(\conv(Y)\bigr)$.
\end{theorem}
Theorem~\ref{strong:caratheodory} is a direct consequence of the standard Carath\'eodory theorem  (see \cite[Theorem~1.1.4]{MR1216521}) and Theorem~\ref{char:relint}. A point $x$ of a convex set $C$ is said to \term{extreme} if there exists no segment $I \subseteq C$ satisfying $x \in \relint(I)$. The set of all extreme points of $C$ is denoted by $\ext(C)$. We emphasize that the notion of extreme point is introduced with respect to sets which are not necessarily closed. For a subset $X$ of $\real^d$ one has
\begin{equation} \label{ext:conv:inclusion}
	\ext \bigl(\conv (X)\bigr) \subseteq X.
\end{equation}
Inclusion \eqref{ext:conv:inclusion} can be derived from Theorem~\ref{strong:caratheodory}. The following characterization of extreme points is given in \cite[Lemma~1.4.6]{MR1216521}.
\begin{lemma} \header{Cap lemma}. \label{cap:lemma} Let $K$ be a closed convex set in $\real^d$ and $p \in K$. Then $p \in \ext(K)$ if and only if  for every open neighborhood $U$ of $p$ there exists a hyperplane $H$ such that $x$ and $K \setminus U$ lie in different open halfspaces defined by $H$. (See also Fig.~\ref{fig:cap:lemma}.)
\end{lemma}

\begin{figtabular}{c}
\unitlength=1.1mm
\begin{picture}(44,20)
	\put(4,-3){\includegraphics[width=40\unitlength]{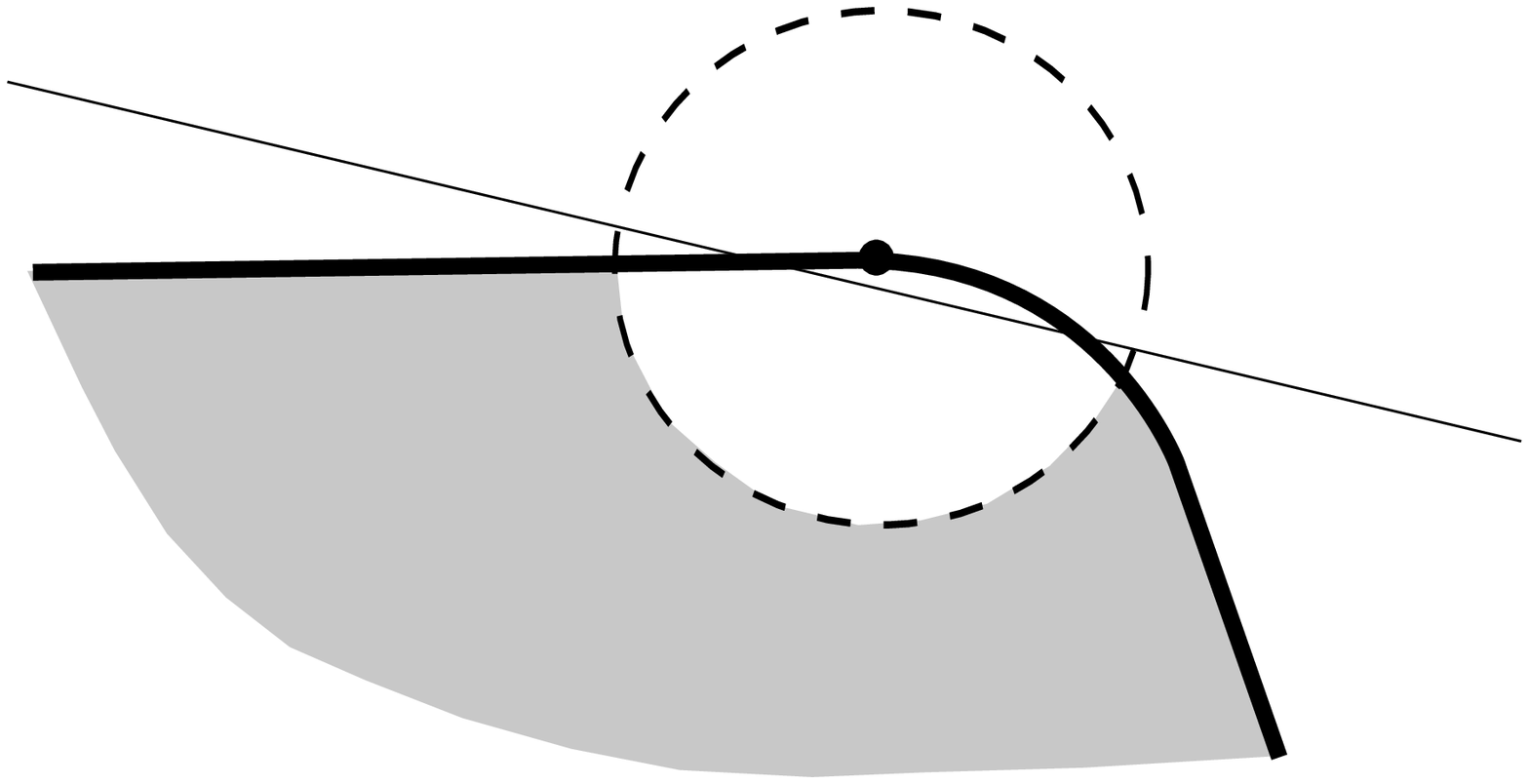}}
	\put(26,12){$p$}
	\put(0,9){$K$}
	\put(5,16){$H$}
	\put(33,16){$U$}
\end{picture}
\\
\parbox[t]{0.90\textwidth}{\caption{\label{fig:cap:lemma} Illustration to Lemma~\ref{cap:lemma} for the case $d=2$. The figure depicts a choice of $H$ in the case that $p \in \ext(K)$ and $U$ is a small neighborhood of $p$. The thick solid line is the boundary of $K$, the dashed line is the boundary of $U$ and the shaded region is $K \setminus U$. In the figure $p$ is chosen to be an endpoint of a one-dimensional face of $K$.}}
\end{figtabular}

 A convex subset $F$ of a closed convex set $K \subseteq \real^d$ is said to be a \term{face} of $K$ if for every segment $I$ lying in $K$ and satisfying $\relintr(I) \cap F \ne \emptyset$ one necessarily has $I \subseteq F$. Every face is necessarily a closed set. Furthermore, directly from the definition it can be seen that if $C$ is a convex subset of $K$ such, $F$ is a face of $K$ and $\relintr(C) \cap \relintr(F) \ne \emptyset$, then $C \subseteq F$. It is known that every closed convex set $K$ in $\real^d$ is the disjoint union of the relative interiors of all faces of $K$ (see \cite[Theorem~2.1.2]{MR1216521}). Furthermore, if $F_2$ is a face of $K$ and $F_1$ is a face of $F_2$, then $F_1$ is a face of $K$. Given an integer $i \ge 0$ by $\ext_i(K)$ we denote the union of all faces of $K$ of dimension at most $i$. The set $\ext_i(K)$ is said to be the \term{$i$-skeleton} of $K$. One has $\ext(K) = \ext_0(K)$. The one-dimensional faces of $K$ which are rays are called \term{extreme rays} of $K$. By $\extr(K)$ we denote the union of all extreme rays of $K$. 

A closed convex set $K$ is said to be \term{line-free} if $K$ does not contain lines. Every closed convex set is a direct sum of a linear space and a line-free closed convex set. By this, in most cases there is no loss of generality in considering line-free closed convex sets only.  The following result can be found in \cite[Theorem~1.4.3 and Corollary~1.4.4]{MR1216521}.

\begin{theorem} \header{Decomposition theorem for convex sets}. \label{decomp:conv:sets}
	Let $K$ be a line-free closed convex set in $\real^d$. Then the following equalities hold.
	\begin{align}
		K & = \conv \bigl(\ext(K) \cup \extr(K)\bigr), \label{mink:union:rep}\\ 
		K & = \conv \bigl(\ext (K)\bigr) + \rec(K). \label{mink:sum:rep}
	\end{align}
\end{theorem}

Given $\rho \ge 0$ and $x \in \real^d$ by $B(x,\rho)$ we denote the closed Euclidean ball of radius $\rho$ centered at $x$.
For closed convex sets $K_1,K_2$ in $\real^d$ the \term{Hausdorff distance} $\delta(K_1,K_2)$ of $K_1$ and $K_2$ is defined by
\[
	\delta(K_1,K_2) := \min \setcond{\rho \ge 0}{K_1 \subseteq K_2+B(o,\rho) \ \text{and} \  K_2 \subseteq K_1+ B(o,\rho)}.
\]
In particular, one has $\delta(K_1,K_2)=+\infty$ if one of the two sets $K_1, K_2$ is empty and the other is not. The Hausdorff distance is a metric on the class of closed convex sets in $\real^d$ (for further information, see \cite[\S\,1.8]{MR1216521}). Some of our results deal with convergence of sequences of closed convex sets. The convergence of such sequences can be introduced in several ways (see \cite[\S\S17.I, \S\S21.VII and \S\S29.VI]{MR0217751}). Under restriction to bounded closed convex sets, the convergence with respect to the Hausdorff distance is the standard choice. In the class of closed convex sets (that is, without boundedness restriction) also other forms of convergence are natural.

A sequence $(K_i)_{i \in \natur}$ of convex sets in $\real^d$ is said to be \term{decreasing} (in the nonstrict sense) if $K_{i+1} \subseteq K_i$ for every $i \in \natur$. Given a decreasing sequence $(K_i)_{i \in \natur}$ as above the set $K:=\bigcap_{i \in \natur} K_i$ is the limit of $(K_i)_{i \in \natur}$ in the sense of the definition given in \cite[\S\S29.VI]{MR0217751}. However, in general, $K$ is not necessarily the limit of $(K_i)_{i \in \natur}$ with respect to the Hausdorff distance. In fact, take $d=2$ and $K_i:= \cone(\{e_1, e_1 + \frac{1}{i} e_2\})$ (where $e_1, e_2$ is the standard basis of $\real^2$). Then $K=\cone(\{e_1\})$, but $\delta(K,K_i)=+\infty$ for every $i \in \natur$.  The following lemma (see \cite[Lemma~1.8.1]{MR1216521}) shows that under boundedness assumption examples as above do not exist.

\begin{lemma} \header{Convergence of a decreasing set sequence}. \label{decr:convergence}
	Let $(K_i)_{i \in \natur}$ be a decreasing sequence of nonempty compact convex sets in $\real^d$. Then $K:= \bigcap_{i=1}^{+\infty} K_i$ is a nonempty compact convex set and $K_i$ converges to $K$ with respect to the Hausdorff distance, as $i \rightarrow +\infty$.
\end{lemma}

\subsection{Polyhedra and maximal lattice-free sets} \label{subsect:polyhedra}

A subset $P$ of $\real^d$ is called a \term{polyhedron} if $P$ is intersection of finitely many closed halfspaces (thus, the empty set is also a polyhedron). In analytic terms, $P \subseteq \real^d$ is a polyhedron if and only if there exist $a_1,\ldots,a_n \in \real^d \setminus \{o\}$ and $\alpha_1,\ldots,\alpha_n \in \real$ with $n \ge 0$ such that 
\begin{equation} \label{P:descr}
	P =\setcond{x \in \real^d}{\sprod{x}{a_i} \le \alpha_i \quad \forall \ i=1,\ldots,n}.
\end{equation}
(In the case $n=0$ one has $P=\real^d$.) Bounded polyhedra are said to be \term{polytopes}. A polyhedron $P$ in $\real^d$ is called \term{rational} if $P$ can be given by \eqref{P:descr} with $a_1,\ldots,a_n \in\rational^d \setminus \{o\}$ and $\alpha_1,\ldots,\alpha_n \in \rational$. A polyhedron $P$ is said to be \term{integral} if $P = \conv(P \cap \integer^d)$. Every integral polyhedron is necessarily also rational. If $P$ is a rational polyhedron, then $\lineal(P)$ and $\rec(P)$ are integral polyhedra. If $P$ is a polyhedron given by \eqref{P:descr} such that $o \in \intr(P)$, $P \ne \real^d$ and $\rec(P)$ is a linear space, then $\alpha_1,\ldots,\alpha_n >0$ and one has
\begin{align} 
	P^\circ = & \conv \left(\setcond{ \frac{a_i}{\alpha_i}}{i=1,\ldots,n}\right), \label{polar:over:hyp:rep} \\
	\|u\|_P  = & \max \setcond{ \frac{\sprod{a_i}{u}}{\alpha_i}}{i=1,\ldots,n}, \label{gauge:over:sup:func}
\end{align}
where $u$ is an arbitrary vector in $\real^d$. If $P$ is a polyhedron, then $\ext(P)$ is precisely the set of all vertices of $P$ and, for $i\in \{0,\ldots,d\}$, $\ext_i(P)$ is the union of all $i$-dimensional faces of $P$. One-dimensional faces of polyhedra are called \term{edges}. Every edge is either a segment or a ray. For a rational $d$-dimensional polyhedron $P$ we denote by $U(P)$ the set of all vectors $u \in \integer^d \setminus \{o\}$ such that $u$ is an outer normal to a facet of $P$ and the components of $u$ are relatively prime integers. One has
\begin{equation} \label{P:repr:U(P)}
	P = \setcond{x \in \real^d}{\sprod{x}{u} \le h(P,u) \quad \forall \ u \in U(P)}.
\end{equation}
The \term{max-facet-width} of a $d$-dimensional rational polyhedron $P$ is defined by 
\[
	\maxfw(P) = \max \setcond{w(P,u)}{u \in U(P)},
\]
where $w(P,\dotvar)$ is the width function, which was defined by \eqref{width:function:def}. It is not hard to see that $\maxfw(P) < \infty$ if and only if $\rec(P)$ is a linear space. By $\Aff(\integer^d)$ we denote the set of all affine transformations $A$ in $\real^d$ satisfying $A(\integer^d) = \integer^d$. Two sets $X$ and $Y$ in $\real^d$ are said to be \term{$\integer^d$-equivalent} if $Y = A(X)$ for some $A \in \integer^d$. If $\cX$ is a class of sets in $\real^d$, then by $\cX / \Aff(\integer^d)$ we denote the set of equivalence classes on $\cX$ with respect to the $\integer^d$-equivalence. The notion of max-facet-width is invariant under $\integer^d$ equivalence, that is,
\[
	\maxfw(P) = \maxfw\bigl(A(P)\bigr)
\]
for every rational polyhedron $P$ in $\real^d$ and every $A \in \Aff(\integer^d)$. The following theorem is well-known (see \cite[Theorem~2]{MR1138580} and \cite{MR1191566}).
\begin{theorem} \label{finiteness:and:volume:boundedness}
	Let $\cL$ be a class of $d$-dimensional integral polytopes in $\real^d$. Then the set $\cL/ \Aff(\integer^d)$ is finite if and only if $\setcond{\vol(L)}{L \in \cL}<+\infty$.
\end{theorem}

A subset $K$ of $\real^d$ is said to be \term{lattice-free} if $K$ is $d$-dimensional closed convex and $\intr(K) \cap \integer^d \ne \emptyset$. A lattice-free set $K$ is said to be \term{maximal lattice-free} if $K$ is not properly contained in another lattice-free set. Proposition~\ref{max-M-free-prop} below shows that maximal $\bM$-free sets (which were defined in the introduction) can be described in terms of maximal lattice-free sets.

\begin{proposition} \header{Description of maximal $\bM$-free sets}. \label{max-M-free-prop}
	Let $d \in \natur$. Let $\bM$ be a mixed-integer space given by \eqref{MI:space}. Then a set $P$ in $\real^d$ is maximal $\bM$-free set if and only if $P=P' \times \real^n$, where $P'$ is a maximal lattice-free set in $\real^m$.
\end{proposition}

The proof of Proposition~\ref{max-M-free-prop} is straightforward and is therefore omitted. The following proposition presents well-known properties of maximal lattice-free sets (see \cite[Propositions~3.1 and 3.3]{MR1114315}).

\begin{proposition} \header{Description of maximal lattice-free sets}. \label{max-lat-free-prop} Let $K$ be a lattice-free set in $\real^d$. Then the following statements hold.
	\begin{enumerate}[I.] 
		\item \label{max-lat-free-description} The set $K$ is maximal lattice-free if and only if $K$ is a polyhedron and the relative interior of each facet of $K$ contains a point of $\integer^d$.
		\item \label{unbounded-max-lat-free} If $K$ is maximal lattice-free and unbounded, then $K$ is $\integer^d$-equivalent to $\real^n \times K'$, where $n \in \{1,\ldots,d-1\}$ and $K'$ is a maximal lattice free set in $\real^{d-n}$.
	\end{enumerate}
\end{proposition}

In view of Propositions~\ref{max-M-free-prop} and \ref{max-lat-free-prop}, every maximal $\bM$-free sets is also maximal lattice-free.

\subsection{Gordan-Dickson lemma}

We shall need the following useful fact (for far-reaching generalizations formulated in the framework of well-quasi orderings see \cite{MR0306057,MR818505,MR818506}).

\begin{lemma} \header{Gordan-Dickson Lemma}. \label{gordan:dickson}
	Let $X \subseteq \natur^d$. Then there exists a finite subset $X'$ of $X$ such that every $x \in X$ satisfies $x' \le x$ for some $x' \in X'$.
\end{lemma}

\section{Main results} \label{sect:results}

In view of Theorem~\ref{decomp:conv:sets} we are motivated to describe the extreme points and the recession cone of $R_L(K)$. This is done in Theorem~\ref{descr:R} under rather weak assumptions on $K$ and $L$. Parts~I and II of Theorem~\ref{descr:R} generalize \cite[Lemma~4.2]{MR2676765} and \cite[Lemmas~2.3 and 2.4]{MR2676765}, respectively.

\begin{theorem} \label{descr:R}  \header{Properties of $L$-reductions}.
	Let $K$ be a line-free closed convex set in $\real^d$ and let $L$ be a $d$-dimensional closed convex set in $\real^d$.  Let $R:=R_L(K)$ and assume $R \ne \emptyset$. Then the following statements hold.
	\begin{enumerate}[I.]
		\item \label{ext:R} For a point $x \in \real^d$ the following conditions are equivalent.
		\begin{enumerate}[(i)]
			\item One has $x \in \ext(R)$.
			\item Either $x \in \ext(K) \setminus \intr(L)$ or there exists a one-dimensional face $I$ of $K$ such that $I \setminus \{x\}$ consists of two connected components $I_1$ and $I_2$ which satisfy $I_1 \subseteq \intr(L)$ and $I_2 \cap L = \emptyset$.
		\end{enumerate}
		\item \label{rep:R} If $\rec(L)$ is a linear space, then $R$ is closed, the recession cones of $R$ and $K$ coincide and, furthermore, one has
		\begin{align}
			R & = \conv \bigl(\ext(R)\bigr) + \rec(K), \label{R:rep:precise} \\
			R & = \conv \bigl(\ext_1(K)\setminus \intr(L)\bigr) + \rec(K). \label{R:rep:less:precise}
		\end{align}
	\end{enumerate}
\end{theorem}

Part~\ref{ext:R} of Theorem~\ref{descr:R} can be illustrated by Fig.~\ref{R_L:illustration} from the introduction. In Fig.~\ref{R_L:illustration} the set $R:=R_L(P)$ is a $7$-gon. It can be seen that the seven vertices of $R_L(P)$ are the only points that satisfy condition (ii) from Part~\ref{ext:R}.

It turns out that, for a general $L$, the functional $R_L$ does not always map polyhedra to polyhedra. The reason of this is that $R_L(K)$ is not always a closed set when $K$ is closed. See Fig.~\ref{fig-two-orhtants} for an example. Our next result characterizes the sets $L$ for which $R_L$ maps polyhedra to polyhedra.
\begin{figtabular}{c}
\unitlength=1mm
\begin{picture}(80,30)
	\put(6,-3){\includegraphics[width=70\unitlength]{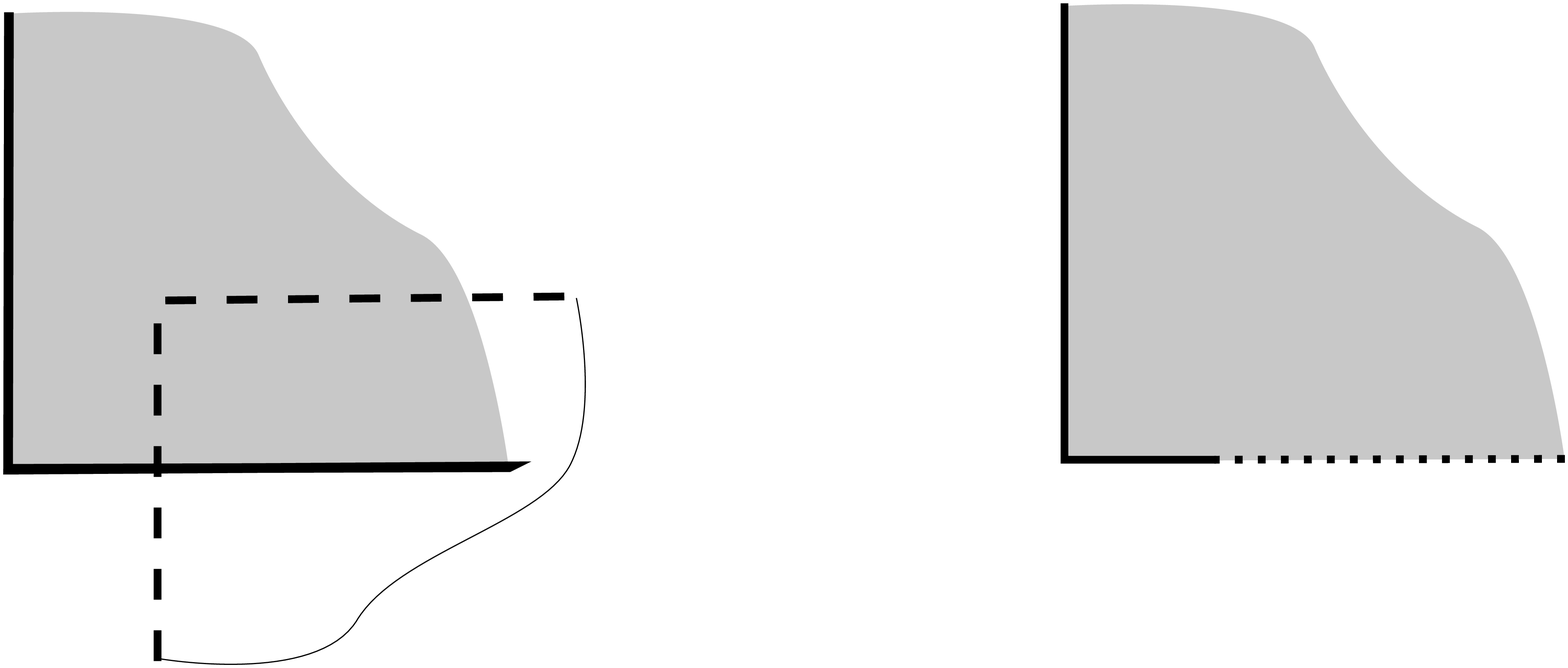}}
	\put(1,20){$K$}
	\put(30,0){$L$}
	\put(40,20){$R_L(K)$}
\end{picture}
\\
\parbox[t]{0.90\textwidth}{\caption{\label{fig-two-orhtants} An example of closed convex sets $K$ and $L$ for which $R_L(K)$ is not closed (for the case $d=2$). Both $K$ and $L$ are translates of orthants. The set $K$ is shaded, the boundary of $L$ is dashed. The dotted line (on the right) is the part of the boundary of $R_L(K)$  which is not contained in $R_L(K)$.}}
\end{figtabular}

\begin{theorem} \label{polyhedrality:characterization} \header{Characterization of $L$-reductions preserving polyhedrality}.
	Let $L$ be a $d$-dimensional closed convex set in $\real^d$. Then the following conditions are equivalent.
	\begin{enumerate}[(i)]
		\item \label{polyh->polyh} $R_L$ maps every polyhedron in $\real^d$ to a polyhedron.
		\item \label{closed->closed} $R_L$ maps every closed convex set in $\real^d$  to a closed convex set.
		\item \label{rec:lin:or:halfspace} Either $L$ is a halfspace or $\rec(L)$ is a linear space.
	\end{enumerate}
\end{theorem}

In Theorem~\ref{polyhedrality:of:closures} below we present a condition on $\cL$ under which $R_\cL$ maps rational polyhedra to rational polyhedra. Theorem~\ref{polyhedrality:of:closures} is a generalization of \cite[Theorem~4.2]{MR2676765}.

\newcommand{\cR}{\mathcal{R}}

\begin{theorem} \label{polyhedrality:of:closures} \header{Condition sufficient for finite generation of an $\cL$-closure}.
	Let $P$ be a rational polyhedron in $\real^d$.  Let $\cL$ be a class of $d$-dimensional rational polyhedra in $\real^d$ such that the following conditions are fulfilled.
	\begin{enumerate}[(a)]
		\item	There exists $k \in \natur$ such that $\maxfw(L) \le k$ for every $L \in \cL$.
		\item There exists $\ell \in \natur$ such that for every $L \in \cL$ and every facet $F$ of $L$ one has $ \bigl(\ell \aff(F)\bigr) \cap \integer^d \ne \emptyset$.
	\end{enumerate}
	Then the following statements hold.
	\begin{enumerate}[I.]
		\item There exists a finite subclass $\cL'$ of $\cL$ such that every $L \in \cL$ satisfies  $R_{L'}(P) \subseteq R_L(P)$ for some $L'\in \cL'$.
		\item One has $R_\cL(P) = R_{\cL'}(P)$ for some finite subclass $\cL'$ of $\cL$.
		\item The set $R_\cL(P)$ is a rational polyhedron.
	\end{enumerate}
\end{theorem}

Part~I is the main assertion of Theorem~\ref{polyhedrality:of:closures}. Part~I resembles the structure of the Gordan-Dickson lemma. In fact, we use the Gordan-Dickson lemma in its proof. As a consequence of Theorem~\ref{polyhedrality:of:closures} we obtain the following result from \cite[Theorem~4.2]{MR2676765}.

\begin{corollary} \label{polyhedrality:of:closures:cor}
	Let $P$ be a rational polyhedron in $\real^d$. Let $\cL$ be a class of $d$-dimensional rational polyhedra in $\real^d$ such that the following conditions are fulfilled.
	\begin{enumerate}[(a)]
		\item There exists $k \in \natur$ such that for every $L \in \cL$ one has $\maxfw(L) \le k$.
		\item Every $L \in \cL$ is a maximal lattice-free set.
	\end{enumerate}
	Then Parts~I-III of Theorem~\ref{polyhedrality:of:closures} are fulfilled.
\end{corollary}

Below we explain the relation of Corollary~\ref{polyhedrality:of:closures:cor} to the Chv\'atal and split closure.
For $u \in \integer^d$ let $\gcd(u)$ denote the greatest common divisor of the components of $u$. A set $L$ in $\real^d$ is called a \term{split} if
\[L=\setcond{x \in \real^d}{i \le \sprod{x}{u} \le i+1},\] where $i \in \integer$, $u \in \integer^d \setminus \{o\}$ and $\gcd(u)=1$ (for an illustration see the set $P_2$ from Fig.~\ref{ex-M-free} in the introduction). A rational polyhedron $L$ in $\real^d$ is a split if and only if $\intr(L) \cap \integer^d = \emptyset$ and $\bd(L) = H_1 \cup H_2$, where $H_1, H_2$ are distinct parallel hyperplanes in $\real^d$ both containing points of $\integer^d$. Let $P$ be a rational polyhedron in $\real^d$. Then the set 
\[
	\Ch(P):= \setcond{x \in P}{\sprod{x}{u} \le \floor{h(P,u)} \ \text{for every} \ u \in \integer^d \setminus \{o\} \ \text{satisfying} \ \gcd(u)=1}
\]
is called the \term{Chv\'atal closure} of $P$ (see \cite{MR0313080}). Clearly, $\Ch(P) = R_\cL(P)$ for $\cL$ consisting of all splits $L$ having the form $L= \setcond{x \in \real^d}{\floor{h(P,u)} \le \sprod{x}{u} \le \floor{h(P,u)} + 1}$ with $u \in \integer^d \setminus \{o\}$ and $\gcd(u)=1$. Thus, Corollary~\ref{polyhedrality:of:closures:cor} implies that for every rational polyhedron $P$ in $\real^d$ the set $\Ch(P)$ is a rational polyhedron (see \cite[Theorem~23.1]{MR874114}). Consider a mixed-integer space $\bM$ (given by \eqref{MI:space}). We call a subset $L$ of $\real^d$ an \term{$\bM$-split} if $L= L' \times \real^n$, where $L'$ is a \term{split} in $\real^m$. If $\cL$ is the class of all $\bM$-splits, then we call the functional $\Sp_\bM=R_\cL$ the \term{split closure} (with respect to $\bM$). By Corollary~\ref{polyhedrality:of:closures:cor}, for every rational polyhedron $P$ in $\real^d$ the split closure $\Sp_\bM(P)$ of $P$ with respect to $\bM$ is a rational polyhedron (see \cite[Theorem~3]{MR1059391}). 

The following theorem shows that, if the condition (b) of Corollary~\ref{polyhedrality:of:closures:cor} is fulfilled, then condition (a) can be reformulated in simple terms.

\begin{theorem} \label{max-fac-width-of-max-lat-free}
	Let $\cL$ be a class of rational polyhedra in $\real^d$ such that each $L \in \cL$ is a maximal lattice-free set. Then the following conditions are equivalent.
	\begin{enumerate}[(i)]
		\item	There exists $k \in \natur$ such that for every $L \in \cL$ one has $\maxfw(L) \le k$.
		\item The set $\cL / \Aff(\integer^d)$ is finite.
	\end{enumerate}
\end{theorem}

Our next result deals with convergence of $R_\cL^i(K)$, as $i \rightarrow +\infty$. Since we consider a decreasing sequence of sets, the convergence is expressed as an intersection.

\begin{theorem} \header{The limit of the sequence of $i$-th $\cL$-closures}. \label{convergence:thm}
	let $\cL$ be a nonempty class of $d$-dimensional closed convex sets in $\real^d$ such that for every $L \in \cL$ the set $\rec(L)$ is a linear space. Let
	\[
		M:=\real^d \setminus \bigcup_{L \in \cL} \intr(L).
	\]
	Then the following statements hold.
	\begin{enumerate}[I.]
		\item	For every line-free closed convex set $K$ in $\real^d$ one has
			\[
				\bigcap_{i=0}^{+\infty} R_\cL^i(K) = \conv(K \cap M) + \rec(K).
			\]
		\item If $M$ is a mixed-integer space (that is, $M=\bM$ with $\bM$ given by \eqref{MI:space}), then for every line-free rational polyhedron $P$ in $\real^d$ one has
	\[
		\bigcap_{i=0}^{+\infty} R_\cL^i(P) = \conv(P \cap \bM).
	\]
	\end{enumerate}
\end{theorem}

Under the given assumptions, Theorem~\ref{convergence:thm} cannot be improved to a result on finite convergence as an example from \cite[Example~2]{MR1059391} shows (but see \cite[Theorem~4]{DelPiaWeismantel10} for conditions sufficient for a finite convergence). 

\begin{corollary} \label{convergence:bounded:case}
	Let $K$, $\cL$ and $M$ be as in Theorem~\ref{convergence:thm} and let $K$ be bounded. Then $R_\cL^i(K)$ converges to $\conv(K \cap M)$ in the Hausdorff distance.
\end{corollary}

Corollary~\ref{convergence:bounded:case} extends Theorem~3 from \cite{MR1814548} which states that the sequence of split closures $\Sp^i_\bM(P)$ converges to $\conv(P \cap \bM)$ in the Hausdorff metric, as $i \rightarrow +\infty$, in the case that $P$ is a rational polytope in $\real^d$ and $\bM$ is a mixed-integer space (given by \eqref{MI:space}). Del Pia and Weismantel showed that the latter result from from \cite{MR1814548} also holds if $P$ is an arbitrary rational polyhedron (see  \cite[Theorem~2]{DelPiaWeismantel10}). We emphasize that, in the case when no boundedness assumptions are made, the convergence with respect to the Hausdorff distance is a stronger assertion than the convergence in the sense of Part~II of Theorem~\ref{convergence:thm}. Thus, compared to Theorem~2 from \cite{DelPiaWeismantel10}, Theorem~\ref{convergence:thm} provides an infinite-convergence result with weaker assumptions and a weaker assertion.

\section{Proofs} \label{sect:proofs}

\begin{proof}[Proof of Theorem~\ref{descr:R}]

Let $\Rbar:= \cl(R)$. First we show the inclusion 
\begin{equation} \label{ext:R:bar:ext:R}
	\ext (\Rbar) \subseteq \ext(R).
\end{equation}
Since $\Rbar \subseteq K$, we have $\ext(\Rbar) \subseteq K$. Let us show that $\ext(\Rbar) \subseteq K \setminus \intr(L)$. Assume the contrary, that is there exists $x \in \ext(\Rbar) \cap \intr(L)$. By the cap lemma (Lemma~\ref{cap:lemma}) applied to $\Rbar$ and the open neighborhood $\intr(L)$ of $x$ there exists a closed halfspace $H^+$ such that $x \not\in H^+$ and $\Rbar \setminus \intr(L) \subseteq H^+$. It follows that $R = \conv\bigl(K\setminus \intr(L)\bigr) \subseteq H^+$ and by this $\Rbar = \cl(R) \subseteq H^+$. On the other hand, $x$ is a point of $\Rbar$ not lying in $H^+$, a contradiction to $\Rbar \subseteq H^+$. Hence \eqref{ext:R:bar:ext:R} is fulfilled.

	\emph{Part~\ref{ext:R}.} Let $x \in \ext(R)$. We show that (ii) is fulfilled. By \eqref{ext:conv:inclusion}, we have $x \in K \setminus \intr(L)$. Hence, by the \emph{separation theorem} (see \cite[Theorem~1.3.4]{MR1216521}), there exists a hyperplane $H$ with $x \in H$ and $H \cap \intr(L) = \emptyset$.  One has $x \in \ext(K \cap H)$. In fact, if $x \not \in \ext(K \cap H)$, then there exist a segment $J \subseteq K \cap H \subseteq K \setminus \intr(L)$ such that $x \in \relintr(J)$. But then $x \not\in \ext(R)$, a contradiction. Let $I$ be the face of $K$ with $x \in \relintr(I)$. We have $0 \le \dim (I) \le 1$. This can be shown arguing by contradiction. If $\dim (I) \ge 2$, then $x \in \relintr(I \cap H)$ and
	\[\dim (I \cap H) \ge \dim (I) -1 \ge 1\]
	contradicting $x \in \ext(K \cap H)$. In the case $\dim (I)=0$, one has $x \in \ext(K)$ and thus $x \in \ext(K) \setminus \intr(L)$. Consider the case $\dim (I)=1$. If $x \not\in \bd(L)$, then $x \not\in L$. Hence there exists a sufficiently small segment $I'$ with $x \in \relintr(I') \subseteq I' \subseteq I \setminus \intr(L) \subseteq K \setminus \intr(L)$. Then $x \not\in \ext(R)$, a contradiction. Thus, $x \in \bd(L)$. The set $I \setminus \{x\}$ consists of two connected components $I_1$ and $I_2$. If both $I_1$ and $I_2$ contain points of $K \setminus \intr(L)$, it follows $x \not\in \ext(R)$, a contradiction to the choice of $x$. Thus, we can assume that $I_1$ or $I_2$ is a subset of $\intr(L)$. Without loss of generality let $I_1 \subseteq \intr(L)$. Then, by convexity of $L$, one has $I_2 \cap L = \emptyset$. This shows (i) $\Rightarrow$ (ii).

	Let us show (ii) $\Rightarrow$ (i). We consider an arbitrary $x \in \real^d$ satisfying (ii). In the case $x \in \ext(K) \setminus \intr(L)$, using the definition of the notion of extreme point, we see that the conditions $x \in R \subseteq K$ and $x \in \ext(K)$ imply (i). Otherwise $x$ lies in the relative interior of a one-dimensional face $I$ of $K$. By assumption, $I \setminus \{x\}$ consists of two connected components $I_1$ and $I_2$ that satisfy $I_1 \subseteq \intr(L)$ and $I_2 \cap L = \emptyset$. Let us verify $x \in \ext(R)$. We argue by contradiction. Assume that $x$ is not an extreme point of $R$, that is, there exist $p_1,p_2 \in R$ such that $p_1 \ne p_2$ and $x \in \relintr([p_1,p_2])$. By Theorem~\ref{strong:caratheodory} for each $i \in \{1,2\}$ there exists an affinely independent set $X_i \subseteq K \setminus \intr(L)$ such that $p_i \in \relintr\bigl(\conv(X_i)\bigr)$. By Theorem~\ref{char:relint} we deduce that $x \in \relintr(\conv(X))$ for $X:=X_1 \cup X_2$. Furthermore, $x$ is not an extreme point of $\conv(X)$ since $x \in \relintr([p_1,p_2])$ and $p_1, p_2 \in \conv(X)$. In view of Theorem~\ref{strong:caratheodory}, there exists an affinely independent set $Y \subseteq \ext\bigl(\conv(X)) \subseteq X \subseteq K \setminus \intr(L)$ with $x \in \relintr \bigl(\conv (Y)\bigr)$. We thus have $x \not\in Y \subseteq K$, $x \in \relint(I)$ and $I$ is a face of $K$. Hence, using properties of faces, we obtain $Y \subseteq I$. The affinely independent set $Y$ lies in a one-dimensional convex set $I$ and satisfies $x \in \relint \bigl(\conv (Y)\bigr)$. It follows that $Y =\{y_1,y_2\}$, where $y_1 \in I_1$ and $y_2 \in I_2$. But $I_1 \subseteq \intr(L)$ and thus $y_1 \in \intr(L)$, a contradiction to $Y \subseteq K \setminus \intr(L)$. This shows $x \in \ext(R)$ and completes the proof of (ii) $\Rightarrow$ (i).

	\emph{Part~\ref{rep:R}.} We assume that $\rec(L)$ is a linear space. Let us show the equality $\rec(K)=\rec(\Rbar)$. Inclusion $\Rbar\subseteq K$ implies $\rec(\Rbar) \subseteq \rec(K)$. In order to verify the reverse inclusion we show
	\begin{align} \label{ray:in:R}
		& p + tu \in R & & \forall \ p \in K \setminus \intr(L) \quad \forall \ u \in \rec(K) \setminus \{o\} \quad \forall \ t \ge 0.
	\end{align}
	We choose arbitrary $p$ and $u$ as in \eqref{ray:in:R}. Consider the convex set $J:=\setcond{p + t u}{t \ge 0} \cap   \intr(L)$. Let us show that $J$ is bounded. Assume the contrary. Since $p \not\in\intr(L)$, we see there exists $t' \ge 0$ such that $p+ t' u \in \bd(L)$. Since $J$ is unbounded, we have $(p+ t' u) + t u \in L$ for every $t \ge 0$. Thus $u \in \rec(L)$. On the other hand, by the convexity of $L$, $(p+t' u) - tu \not\in L$ for every $t>0$. Thus, $u \ne o$, $u \in \rec(L)$, $-u \not\in\rec(L)$. It follows that $\rec(L)$ is not a linear space, a contradiction to the choice of $L$. The latter shows that $J$ is bounded. The boundedness of $J$ implies $p + s u  \in K \setminus \intr(L)$ for all sufficiently large $s \ge 0$. Thus, every $p+ t u$ with $t \ge 0$ is a convex combination of $p \in K \setminus \intr(L)$ and $p+s u \in K \setminus \intr(L)$ with an appropriate $s \ge 0$. It follows $p+ tu \in R$. Hence \eqref{ray:in:R} is fulfilled.  Since $K \setminus \intr(L) \ne \emptyset$, \eqref{ray:in:R} implies $\rec(K) \subseteq \rec(\Rbar)$. Thus we have verified $\rec(\Rbar) = \rec(K)$.

	Let us show that $R$ is closed, that is, $R=\Rbar$. The inclusion $R \subseteq \Rbar$ is trivial. In view of  \eqref{mink:union:rep} for showing $\Rbar \subseteq R$ it is sufficient to verify the inclusions $\ext(\Rbar) \subseteq R$ and $\extr(\Rbar) \subseteq R$. The inclusion $\ext(\Rbar) \subseteq R$ follows from \eqref{ext:R:bar:ext:R}. For showing $\extr(\Rbar) \subseteq R$  we consider an arbitrary extremal ray $I$ of $\Rbar$. Let $x$ be the endpoint of $I$. Since $x$ is an extreme point of $I$ and $I$ is a face of $\Rbar$, it follows that $x \in \ext(\Rbar)$. Then \eqref{ext:R:bar:ext:R} implies  $x \in \ext(R)$. We represent $I$ by $I = \setcond{x+ t u}{t \ge 0}$, where $u \in \rec(\Rbar)$. Since $\rec(\Rbar) \subseteq \rec(K)$, by \eqref{ray:in:R} one has $x+ tu \in R$ for every $t \ge 0$. That is, $I \subseteq R$. We have verified $\Rbar \subseteq R$. Thus, $R=\Rbar$. Equality \eqref{R:rep:precise} follows from \eqref{mink:sum:rep}. Equality \eqref{R:rep:less:precise} is a consequence of \eqref{R:rep:precise} and Part~I.
\end{proof}

We shall use the following two lemmas.

\begin{lemma} \label{no:reduction:case}
	Let $K$ and $L$ be closed convex sets in $\real^d$ such that  $\lineal(K) \not\subseteq \rec(L) \cup (-\rec(L))$. Then $R_L(K) = K$.
\end{lemma}
\begin{proof}
	We choose $u \in \lineal(K)$ such that both $u$ and $-u$ do not belong to $\rec(L)$.
	Consider an arbitrary $x \in K$. Let $I$ be the line through $x$ parallel to $u$. By the choice of $u$, the intersection $I \cap L$ is bounded. Hence $x \in I=\red_L(I) \subseteq \red_L(K)$. We have shown $K \subseteq \red_L(K)$. The reverse inclusion is trivial.
\end{proof}

The proof of the following lemma is straightforward (and is therefore omitted).

\begin{lemma} \label{projection:lemma}
	Let $K, L$ be closed convex sets in $\real^d$ having the form $K= K' \times \real^n$ and $L= L' \times \real^n$ where $K'$ and $L'$ are closed convex sets in $\real^m$ and $m, n \ge 0$ are integers with $m+n=d$. Then $R_L(K) = R_{L'}(K') \times \real^n$.
\end{lemma}

\begin{proof}[Proof of Theorem~\ref{polyhedrality:characterization}]
	Let us show \eqref{rec:lin:or:halfspace} $\Rightarrow$ \eqref{closed->closed}. Assume that \eqref{rec:lin:or:halfspace} is fulfilled. If $L$ is a halfspace, \eqref{closed->closed} is trivial. Thus, we assume that $\rec(L)$ is a linear space. We consider an arbitrary closed convex set $K$ and show that $R_L(K)$ is closed. The case $R_L(K) = \emptyset$ is trivial. Thus, let $R_L(K) \ne \emptyset$. If $\lineal(K) \not\subseteq \rec(L)$, then by Lemma~\ref{no:reduction:case}, we have $R_L(K) = K$ and thus $R_L(K)$ is closed. Consider the case $\lineal(K) \subseteq \rec(L)$. If $K$ is line-free the assertion follows from Theorem~\ref{descr:R}.\ref{rep:R}. If $K$ is not line-free, then choosing an appropriate coordinate system in $\real^d$ we can assume that $K$ and $L$ can be given as in Lemma~\ref{projection:lemma} with $K'$ being line-free. Thus, this case can be reduced to the case of a line-free $K$.

	Let us show \eqref{rec:lin:or:halfspace} $\Rightarrow$ \eqref{polyh->polyh}. Assume that \eqref{rec:lin:or:halfspace} is fulfilled. Let $P$ be an arbitrary polyhedron in $\real^d$. It suffices to consider the case $R_L(P) \ne \emptyset$. If $\lineal(P) \not\subseteq \rec(L)$ we argue in the same way as in the proof of the implication \eqref{rec:lin:or:halfspace} $\Rightarrow$ \eqref{closed->closed} and obtain $R_L(P)=P$. If $P$ is line-free we apply Theorem~\ref{descr:R}. By Theorem~\ref{descr:R}.\ref{rep:R}, $R_L(P)$ is a closed set and $\rec\bigl(R_L(P)\bigr) = \rec(P)$. Thus, $\rec\bigl(R_L(P)\bigr)$ is a polyhedral cone. By Theorem~\ref{descr:R}.\ref{ext:R}, every extreme point of $R_L(P)$ is either a vertex of $P$ or lies in the relative interior of an edge of $P$. Furthermore, Theorem~\ref{descr:R}.\ref{ext:R} yields that the relative interior of every edge of $P$ contains at most one extreme point of $R_L(P)$. Thus, the number of extreme points of $R_L(P)$ is finite. Taking into account \eqref{mink:sum:rep}, it follows that $R_L(P)$ is a Minkowski sum of a polytope and a polyhedral cone. Hence $R_L(P)$ is a polyhedron. The case that $\rec(P) \subseteq \rec(L)$ and $P$ is not line-free can be reduced to the case of line-free sets $P$ with the help of  Lemma~\ref{projection:lemma}.

	 Let us verify the implications (i) $\Rightarrow$ (iii) and (ii) $\Rightarrow$ (iii). Assume that (iii) is not fulfilled. Let us show that then (i) and (ii) are also not fulfilled. We represent $\rec(L)$ as $\rec(L) = X + C$, where $X$ is a linear space, $C$ is a line-free closed convex cone and $X \cap \lin(C) = \{o\}$. One has $\dim (X) + \dim(C) = \dim \bigl(\rec(L)\bigr)$. Since (iii) is not fulfilled, either $\dim(C)=1$ and $\dim \bigl(\rec(L)\bigr)<d$ or $\dim(C) \ge 2$.  In the case $\dim(C)=1$, we choose an arbitrary $u_1 \in C \setminus \{o\}$ and an arbitrary $u_2 \in \real^d \setminus \lin\bigl(\rec(L)\bigr)$. In the case $\dim(C) \ge 2$ we consider a two-dimensional linear space $Y$ such that $Y \cap C$ is two dimensional and choose $u_1,u_2$ to be a basis of $Y$ such that $Y \cap C = \cone \bigl(\{u_1,u_2\}\bigr)$. We fix $p \in \intr(L)$. Since $-u_1 \not\in \rec(L)$ and $p \in \intr(L)$, there exists $t \ge 0$ such that $p- tu_1 \in \bd(L)$. We define $K := p- 2 t u_1 + \cone (\{u_1,-u_2\})$. For an illustration to the following arguments see Fig.~\ref{fig-constructing-non-closed}. Let $R:=R_L(K)$. By Theorem~\ref{descr:R}.\ref{ext:R}, we have $p - t u_1 \in \ext(R)$. It follows that $p \not\in R$ since otherwise one had $p- t u_1 = \frac{1}{2} (p- 2 t u_1) + \frac{1}{2} p$ with both $p$ and $p-2 t u_1$ belonging to $R$, which would yield a contradiction to $p-t u_1 \in \ext(R)$. On the other hand, for every $\varepsilon >0$ on has $p - \varepsilon u_2 \in R$. This is shown as follows. Since $t u_1 - \varepsilon u_2 \not \in \rec(L)$, the set $\setcond{ p - t u_1 + s(t u_1 - \varepsilon u_2)}{s \ge 0} \cap \intr(L)$ is bounded. Furthermore, from the definition of $K$ it follows $p-t u_1 + s (t u_1 - \varepsilon u_2) \in K$ for every $s \ge 0$. Therefore one can choose a sufficiently large $s > 1$ such that the point $p - t u_1 + s (t u_1 - \varepsilon u_2)$ belongs to $K \setminus \intr(L)$. The latter implies $p- \varepsilon u_2 \in R$ since 
	\[
		p - \varepsilon u_2 = \Bigl(1- \frac{1}{s}\Bigr) (p-t u_1) + \frac{1}{s} \Bigl( p - t u_1 + s (tu_1 - \varepsilon u_2)\Bigr),
	\] 
	that is, $p- \varepsilon u_2$ is a convex combination of the points $p- tu_1$ and $p -t u_1 + s (t u_1 - \varepsilon u_2)$ both belonging to $R$. Since $\varepsilon>0$ is arbitrary, we conclude that $p$ belongs to the closure of $R$ but not to $R$. Thus, assuming that (iii) is not fulfilled we have shown that (ii) is not fulfilled. Since $K$ is a polyhedron and $R$ is not a polyhedron, also (i) is not fulfilled.
\end{proof}

\begin{figtabular}{c}
\unitlength=1mm
\begin{picture}(110,45)
	\put(15,-3){\includegraphics[width=80\unitlength]{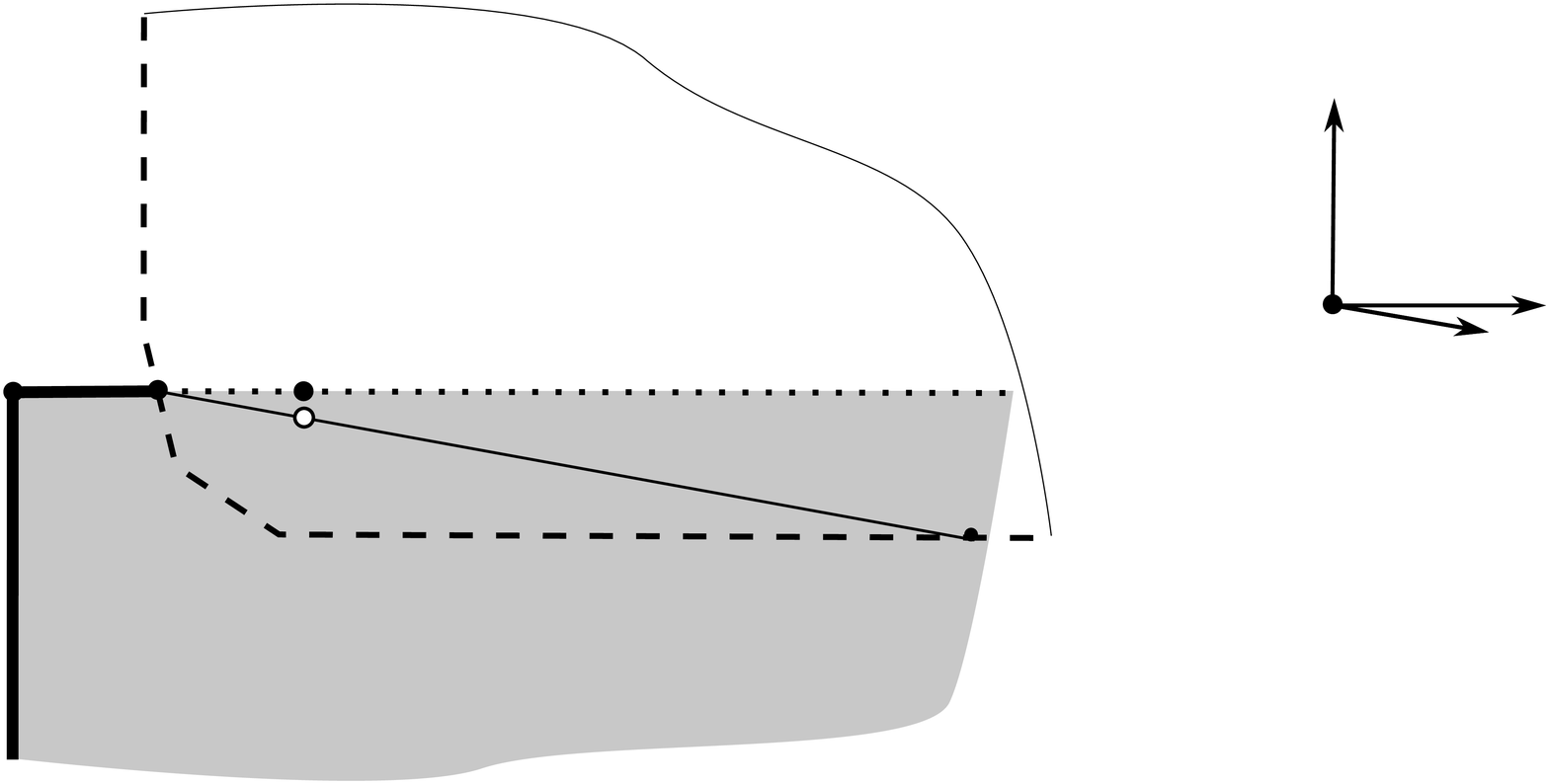}}
	\put(2,18){$p-2 t u_1$}
	\put(30,20){$p$}
	\put(95,22){$u_1$}
	\put(85,32){$u_2$}
	\put(10,4){$K$}
	\put(18,27){$L$}
	\put(87,16){$t u_1 - \varepsilon u_2$}
\end{picture}
\\
\parbox[t]{0.90\textwidth}{\caption{\label{fig-constructing-non-closed}Illustration to the proof of the implications (i) $\Rightarrow$ (iii) and (ii) $\Rightarrow$ (iii) of Theorem~\ref{polyhedrality:characterization} for the case $d=2$ and $\dim (C)=2$. The dashed line is the boundary of $L$, the shaded region is $K$ and the white dot is the point $p-\varepsilon u_2$.}}
\end{figtabular}

The following two lemmas will be used in the proof of Theorem~\ref{polyhedrality:of:closures}.

\begin{lemma} \label{integrality:lemma}
	Let $L$ be a $d$-dimensional rational polyhedron in $\real^d$ with $\maxfw(L)<\infty$ and let $p \in \rational^d \cap \intr(L)$. Assume that there exist constants $k, \ell, m \in \natur$ such that the following conditions are fulfilled.
	\begin{enumerate}[(a)]
		\item $\maxfw(L) \le k$.
		\item For every facet $F$ of $L$ one has $(\ell \aff(F)) \cap \integer^d \ne \emptyset$.
		\item $m \cdot p \in \integer^d$.
	\end{enumerate}
	Then the following statements hold.
	\begin{enumerate}[I.]
		\item $(k \cdot \ell \cdot m)! \cdot \ext\bigl((L-p)^\circ\bigr) \subseteq \integer^d$.
		\item For every $z \in \integer^d$, the value $(k \cdot \ell \cdot m)! \cdot \|z \|_{L-p}$ is a nonnegative integer.
	\end{enumerate}		
\end{lemma}
\begin{proof}
	The assumption $\maxfw(L) < +\infty$ implies that $\rec(L)$ is a linear space.
	Thus, by \eqref{polar:over:hyp:rep} and \eqref{P:repr:U(P)} we have
	\begin{equation} \label{norm:z:L-p}
		(L-p)^\circ = \conv \left( \setcond{\frac{u}{h(L-p,u)}}{u \in U(L)} \right).
	\end{equation}
	Consider an arbitrary $u \in U(L)$. One has 
	\begin{align}
		& u \in  \integer^d, \nonumber \\
		& h(L-p, u) =  h(L,u) - \sprod{p}{u} \in \frac{1}{\ell} \integer - \frac{1}{m} \integer \subseteq \frac{1}{\ell m} \integer, \label{in:lattice} \\
		& h(L-p,u) \le h(L-L,u) = w(L,u) \le k, \label{not:large} \\
		& h(L-p,u) > 0. \nonumber 
	\end{align}
	Above \eqref{in:lattice} follows from (b) and (c), and \eqref{not:large} follows from (a).
	Thus, $\ell \cdot m \cdot h(L-p,u) \in \natur$ and $\ell \cdot m \cdot h(L-p,u) \le k \cdot \ell \cdot m$. We conclude that $(k \cdot \ell \cdot m)!$ is divisible by $\ell \cdot m \cdot h(L-p,u)$. By this, taking into account \eqref{norm:z:L-p}, we derive Part~I. Part~II follows from Part~I in view of \eqref{gauge:sup:f:eq}.
\end{proof}

Let $P$ be a nonempty line-free rational polyhedron in $\real^d$ and $L$ be a $d$-dimensional closed convex set in $\real^d$ such that $\rec(L)$ is a linear space. Below we introduce a matrix $\Rem_L(P)$, which we call the \term{remainder matrix} of $P$ with respect to $L$. Let $V$ be the set of all vertices and $E$ the set of all edges of $P$. For $v \in V$ and $e \in E$ with $v \in e$ let $u(v,e) \in \integer^d$ be a vector with $e \subseteq \setcond{v + t u(v,e)}{t \ge 0}$ and such that the greatest common divisor of the components of $u(v,e)$ is equal to one. Then 
\[
	\Rem_L(P) := \bigl(r_L(v,e)\bigr)_{v \in V, e \in E} \in \real^{V \times E},
\]
where the entries $r_L(v,e)$ are defined as follows. In the case $v \not\in e$ we let $r_L(v,e) := 0$. In the case $v \in e$ we define $r_L(v,e)$ by the following formula.
\begin{equation*}
	r_L(v,e) := 
	\begin{cases}
		0, & e \subseteq \intr(L), \\
		+\infty, & v \not\in \intr(L), \\
		\|u(v,e)\|_{L-v}, & v \in \intr(L), e \not\subseteq \intr(L).
	\end{cases}
\end{equation*}
The definition of the remainder matrix is motivated by \eqref{R:rep:less:precise}. In fact, if we keep the polyhedron $P$ fixed and vary $L$, then (in view of \eqref{R:rep:less:precise}) the `size' of the set $R_L(P)$ can be expressed in terms of the `size' of $\ext_1(P) \setminus \intr(L)$. The matrix $\Rem_L(P)$ is a quantitative representation for the 'size' of $\ext_1(P) \setminus \intr(L)$. More precisely, if we consider $v \in V$ and $e \in E$ with $v \in E$, then the entry $r_L(v,e)$ has the following interpretation. The largest possible value for $r_L(v,e)$ is $r_L(v,e)=+\infty$, which indicates that $v$ remains in $\ext_1(P) \setminus \intr(L)$. The smallest possible value $r_L(v,e)=0$ indicates that $v$ and the whole edge $e$ incident to $v$ is `cut off' by $L$ completely. In other cases one has $0< r_L(v,e) < +\infty$, which indicates that $v$ does not remain in $\ext_1(P) \setminus \intr(L)$ but a part of $e$ remains in $\ext_1(P) \setminus \intr(L)$, while the value $r_L(v,e)$ represents how large this part of $e$ is. See also Fig.~\ref{fig-illustration-rem-mx} for an illustration.

\begin{figtabular}{c}
\unitlength=1.2mm
\begin{picture}(115,32)
	\put(0,8){\includegraphics[width=110\unitlength]{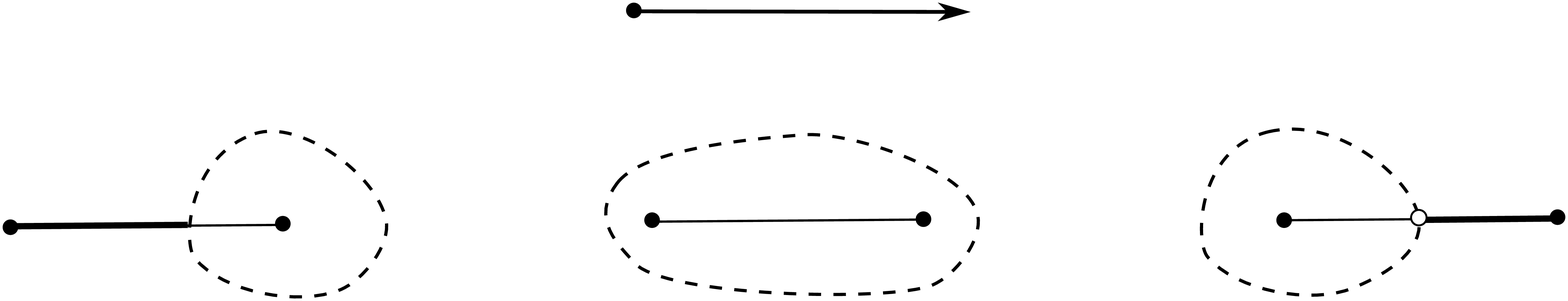}}
	\put(0,14){$v$}
	\put(10,11){$e$}
	\put(25,18){$L$}
	\put(2,0){$r_L(v,e)=+\infty$}
	\put(45,14){$v$}
	\put(55,11){$e$}
	\put(65,18){$L$}
	\put(43,0){$r_L(v,e)=0$}
	\put(90,14){$v$}
	\put(95,11){$e$}
	\put(97,18){$L$}
	\put(83,0){$r_L(v,e)=\|u(v,e)\|_{L-v}>0$}
	\put(65,25){$u(v,e)$}
\end{picture}
\\
\parbox[t]{0.90\textwidth}{\caption{\label{fig-illustration-rem-mx}Illustration to the definition of the remainder matrix for the case $d=2$. The figure depicts a vertex $v$ and an edge $e$ (of some polyhedron $P$) with $v \in e$ and three different choices of $L$. The dashed line is the boundary of $L$. The thick line is the set $e \setminus \intr(L)$. The white dot in the case $r_L(v,e)=\|u(v,e)\|_{L-v}$ (on the right) is the point $v+ \frac{u(v,e)}{\|u(v,e)\|_{L-v}}$ of the intersection of $e$ and $\bd(L)$.}}
\end{figtabular}

For matrices of a given size we define the partial order $\le$ in the same way as for elements of $\real^d$ (that is, by comparison of the respective components).

\begin{lemma} \label{how:much:remains}
	Let $P$ be a nonempty rational line-free polyhedron in $\real^d$ and let $L$ and $L'$ be $d$-dimensional closed convex sets in $\real^d$ such that $\rec(L)$ and $\rec(L')$ are linear spaces and the relation $\Rem_L(P) \le \Rem_{L'}(P)$ is fulfilled. Then $R_L(P) \subseteq R_{L'}(P)$.
\end{lemma}
\begin{proof}
In view of \eqref{R:rep:less:precise} one has 
\[
	R_L(P) = \conv \left(\bigcup_{e \in E} R_L(e) \right)+\rec(P),
\]
where $E$ denotes the set of all edges of $P$. The same equality also holds with $L'$ in place of $L$. Thus, it suffices to show that the relation $\Rem_L(P) \le \Rem_{L'}(P)$ implies $R_L(I) \subseteq R_{L'}(I)$ for every $e \in E$.

Consider an arbitrary edge $e \in E$.  If $e$ is a segment with endpoints $v$ and $w$, then the following implications hold. 
\begin{align}
	& v \not\in \intr(L), w \not\in \intr(L) & &\Longrightarrow & &R_L(e)=e, \nonumber \\ 
	& & & & & r_L(v,e)=r_L(w,e)=+\infty. \label{all:remains} \\
	& v \in \intr(L), w \in \intr(L) & &\Longrightarrow & &R_L(e)=\emptyset, \nonumber \\
	& & & & & r_L(v,e)=r_L(v,e)=0. \label{all:cut:off} \\
	& v \in \intr(L), w \not\in \intr(L) & &\Longrightarrow & & R_L(e) = \left[v + \frac{u(v,e)}{\|u(v,e)\|_{L-v}},w\right], \nonumber \\
	& & & & & r_L(v,e)= \|u(v,e)\|_{L-v}>0, \nonumber \\ 
	& & & & &  r_L(w,e) = +\infty. \label{part:cut:off}
\end{align}
The implications \eqref{all:remains} and \eqref{all:cut:off} have straightforward proofs. The implication \eqref{part:cut:off} follows from Proposition~\ref{ray-prop}. If $e$ is a ray and $v$ is the endpoint of $e$, then the following implications hold.
\begin{align}
	&v \not\in \intr(L) & &\Longrightarrow & &R_L(e)=e, \nonumber \\
	& & & & & r_L(v,e) = +\infty. \label{all:remains:rays} \\	
	&e \subseteq \intr(L) & &\Longrightarrow & &R_L(I)=\emptyset, \nonumber \\
	& & & & & r_L(v,e)=0. \label{nothing:remains:rays} \\
	&v \in\intr(L), e \not\subseteq \intr(L) & &\Longrightarrow &  &R_L(e)= \setcond{v+ \frac{u(v,e)}{\|u(v,e)\|_{L-v}} + t u(v,e)}{t \ge 0} \nonumber \\
	& & & & & r_L(v,e)=\|u(v,e)\|_{L-v} > 0. \label{part:remains:rays}
\end{align}

In \eqref{all:remains:rays} the equality $R_L(e)=e$ follows by applying  Theorem~\ref{descr:R} to the one-dimensional polyhedron $e$.  In fact, by Theorem~\ref{descr:R}, the recession cones of $R_L(e)$ and $e$ coincide and $v$ is a vertex of $R_L(e)$. Hence $R_L(e)=e$. The implication \eqref{nothing:remains:rays} is trivial. The implication \eqref{part:remains:rays} follows directly from Proposition~\ref{ray-prop}.

Let $e$ be an arbitrary edge of $P$. We shall use the implication \eqref{all:remains}--\eqref{part:remains:rays} given above to show $R_L(e) \subseteq R_{L'}(e)$. Consider the case that $e$ is a segment and let $v$ and $w$ be the endpoints of $e$. If both $r_L(v,e)$ and $r_{L'}(w,e)$ are zero or both $r_L(v,e)$ and $r_{L'}(w,e)$ are $+\infty$ one has $r_L(e)=\emptyset$ or $r_{L'}(e)=e$ and the inclusion $r_L(e) \subseteq r_{L'}(e)$ is fulfilled. Otherwise, one of the two values $r_L(v,e), r_L(w,e)$ is finite and strictly positive and the other one is $+\infty$ and the same is also valid for $L'$ in place of $L$. Without loss of generality let $r_L(w,e) = +\infty$. Then, by $r_L(w,e) \le r_{L'}(w,e)$ one has $r_{L'}(w,e) = + \infty$ and hence $r_L(v,e)$ and $r_{L'}(v,e)$ are positive and finite. It follows $\|u(v,e)\|_{L-v} = r_L(v,e) \le r_{L'}(v,e) = \|u(v,e)\|_{L'-v}$. Representing $R_L(e)$ and $R_{L'}(e)$ as in the implication \eqref{part:cut:off} we deduce $R_L(e) \subseteq R_{L'}(e)$.

The proof of $R_L(e) \subseteq R_{L'}(e)$ for the case that $e$ is a ray is analogous and relies on the implications \eqref{all:remains:rays}--\eqref{nothing:remains:rays}.
\end{proof}

\begin{proof}[Proof of Theorem~\ref{polyhedrality:of:closures}]
 Without loss of generality we assume $P \ne \emptyset$. If for some $L \in \cL$ one has $\lineal(P) \not\subseteq \lineal(L)$, then by Lemma~\ref{no:reduction:case} one has $R_L(P)=P$. Thus, without loss of generality we can assume that $\lineal(P) \subseteq \lineal(L)$ for every $L \in \cL$. It suffices to consider the case that $P$ is line-free. In fact, if $P$ is not line-free, then applying an appropriate transformation from $\Aff(\integer^d)$ to $P$ and to all elements of $\cL$ we can assume that $P=P' \times \real^n$, where $P'$ is a line-free rational polyhedron. In view of Lemma~\ref{projection:lemma} we can apply a canonical projection to $P$ and to all elements of $\cL$ passing to the case of a line-free polyhedron $P$. Thus, assume that $P$ is line-free.

In view of Lemma~\ref{integrality:lemma}, for every $L \in \cL$ each component of $(k \cdot \ell \cdot m)! \cdot \Rem_L(P)$ is either a nonnegative integer or $+\infty$. Let us apply the Gordan-Dickson lemma to the set of matrices $(k \cdot \ell \cdot m)! \cdot \Rem_L(P)$, with $L \in \cL$, which are partially ordered by $\le$. We conclude that one can choose a finite subclass $\cL'$ of $\cL$ such that for every $L \in \cL$ there exists $L'\in \cL'$ with $\Rem_{L'}(P) \le \Rem_L(P)$. Above we applied the Gordan-Dickson lemma in the case  where some of the entries can be $+\infty$. Note that the standard form of the Gordan-Dickson lemma (Lemma~\ref{gordan:dickson}), in which the entries are assumed to belong to $\natur$, can be easily extended to this slightly more general case. By Lemma~\ref{how:much:remains}, for $L'\in \cL'$ and $L \in \cL$ the relation $\Rem_{L'}(P) \le \Rem_L(P)$ implies $R_{L'}(P) \subseteq R_L(P)$. This shows Part~I. Part~II is a straightforward consequence of Part~I.
Assumption (a) implies that for every $L \in \cL$ the set $\rec(L)$ is a linear space. Hence by Theorem~\ref{polyhedrality:characterization}, for every $L \in \cL$, the set $R_{L}(P)$ is a polyhedron. Using Theorem~\ref{descr:R} we see that the polyhedra $R_L(P)$ with $L \in \cL$ are rational. In fact, taking into account Part~I of Theorem~\ref{descr:R} and the rationality of $L$ and $P$, we have $\ext(\bigl(R_L(P)\bigr) \subseteq \rational^d$. By Part~II of Theorem~\ref{descr:R}, we have $R_L(P)= \conv\bigl(\ext(R_L(P)\bigr)+\rec(P)$. Since $\rec(P)$ is a rational polyhedron, we deduce that $R_L(P)$ is a Minkowski sum of two rational polyhera. It follows that $R_L(P)$ is a rational polyhedron. For the finite set $\cL'$ defined above we have $R_\cL(P)= \bigcap_{L \in \cL'} R_L(P)$. Thus, $R_\cL(P)$ is intersection of finitely many rational polyhedra. It follows that $R_\cL(P)$ is a rational polyhedron, that is, Part~III is fulfilled.
\end{proof}

\begin{remark} \header{Non-redundancy of the assumptions (a) and (b) in Theorem~\ref{polyhedrality:of:closures}}. Below we give examples which show that neither assumption (a) nor assumption (b) alone are enough for deriving the assertion of Theorem~\ref{polyhedrality:of:closures}.

Let us first give an example showing that assumption (a) alone is not enough. Let $K$ be the intersection of $[0,1]^2$ with a circular disk of radius $1/2$ centered at the origin (thus, $K$ is a circular sector). For rational values $0 < t < 1/2$ and $0 < s < 1/2$ such that the line through $(t,1/2)^\top$ and $(1/2,s)^\top$ does not meet $K$ we define $L_{s,t}:=[t,1] \times [s,1]^d$. Let $\cL$ be the set of all $L_{s,t}$ as described above (see also Fig.~\ref{fig-ex-b-necessary}). Then $\cL$ satisfies the assumption (a) of Theorem~\ref{polyhedrality:of:closures} since $\maxfw(L) \le 1$ for every $L \in \cL$. By construction $R_\cL([0,1/2]^2)=K$.  Thus, $[0,1/2]^2$ is a polyhedron and $R_\cL([0,1/2]^2)$ is not a polyhedron.

\begin{figtabular}{cc}
\unitlength=1.4mm
\begin{picture}(43,35)
	\put(6,-2){\includegraphics[width=35\unitlength]{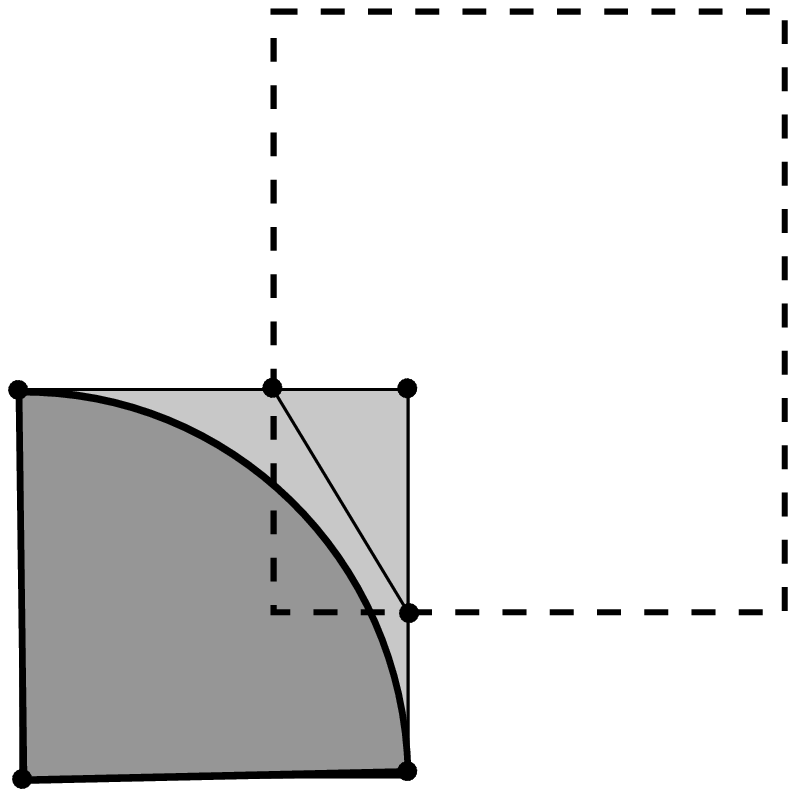}}
	\put(8,17){$(t,1/2)^\top$}
	\put(25,3){$(1/2,s)^\top$}
	\put(12,5){$K$}
	\put(42,20){$L_{s,t}$}
	\put(41,33){$(1,1)^\top$}
	\put(4,-3){$o$}
	\put(21,17){$(1/2,1/2)^\top$}
\end{picture}
&
\unitlength=1mm
\begin{picture}(50,40)
	\put(6,-3){\includegraphics[width=35\unitlength]{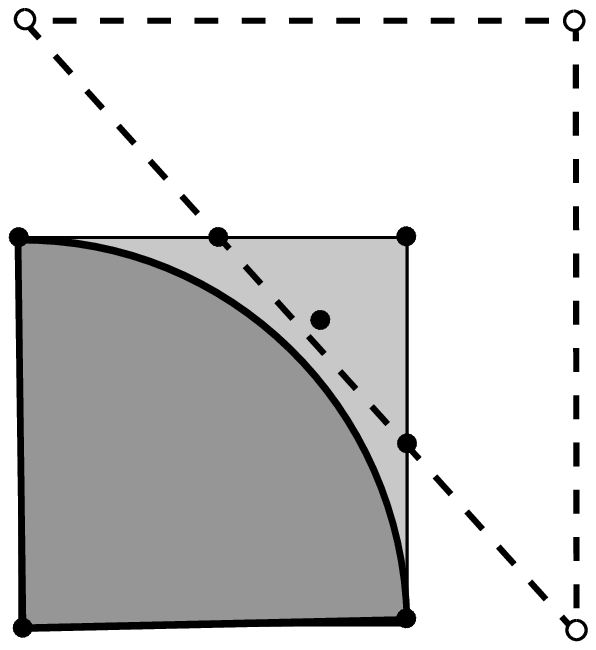}}
	\put(12,6){$K$}
	\put(42,24){$L_{p,q}$}
	\put(26,17){$p$}
	\put(5,36){$q$}
	\put(41,-2){$r$}
	\put(41,36){$(r_1,q_2)^\top$}
	\put(4,-4){$o$}
\end{picture}

\\
\parbox[t]{0.45\textwidth}{\caption{\label{fig-ex-b-necessary} Example showing that in Theorem~\ref{polyhedrality:of:closures} the assumption (a) alone is not enough for deriving the assertion. The dashed line is the boundary of $L_{s,t}$.}}
&
\parbox[t]{0.45\textwidth}{\caption{\label{fig-ex-a-necessary} Example showing that in Theorem~\ref{polyhedrality:of:closures} the assumption (b) alone is not enough for deriving the assertion. The dashed line is the boundary of $L_{p,q}$.}}
\end{figtabular}

Let us give an example showing that assumption (b) alone is not enough. Again, let $d=2$ and let us use $K$ from the previous example. For every rational point $p \in [0,1/2]^d \setminus K$ there exist points $q=(q_1,q_2)^\top \in \integer^2$ and $r=(r_1,r_2)^\top \in \integer^2$ such that $q_1<r_1$, $K$ and $p$ lie in different open halfspaces defined by the line $\aff(\{q,r\})$ and, furthermore, both $q$ and $r$ do not belong to $[0,1/2]^2$. For $q$ and $r$ as above we define the triangle $L_{q,r}$ with vertices $q$, $r$ and $(r_1,q_2)^\top$. Let $\cL$ be the set of all $L_{q,r}$ as described above (see also Fig.~\ref{fig-ex-a-necessary}). Then $\cL$ satisfies the assumption (b) of Theorem~\ref{polyhedrality:of:closures} (with $\ell=1$) and $R_\cL([0,1/2]^2) = K$. Thus, also for this example $[0,1/2]^2$ is a polyhedron and $R_\cL([0,1/2]^2)$ is not a polyhedron. \eop
\end{remark}

\begin{proof}[Proof of Corollary~\ref{polyhedrality:of:closures:cor}]
In view of Proposition~\ref{max-lat-free-prop}.\ref{max-lat-free-description}, condition (b) of Corollary~\ref{polyhedrality:of:closures:cor} implies condition (b) of Theorem~\ref{polyhedrality:of:closures} with $\ell=1$.  Thus, the assumptions of Theorem~\ref{polyhedrality:of:closures} are fulfilled and the assertion follows.
\end{proof}

The following proposition will be used in the proof of Theorem~\ref{max-fac-width-of-max-lat-free}.

\begin{proposition} \label{vol:versus:maxfw}
	Let $P$ be a $d$-dimensional rational polytope in $\real^d$. Then
	\begin{equation} \label{vol:vs:maxfw:ineq}
		\vol(P) \le \maxfw(P)^d.
	\end{equation}
\end{proposition}
\begin{proof}
	Consider $u \in U(P)$ with $\maxfw(P)=w(P,u).$ Let $k:=\maxfw(P)$. Then there exist linearly independent $u_1,\ldots,u_d \in U(P)$ such that $u_1=u$. For the parallelotope
	\[
		P':=\setcond{x \in \real^d}{ h(P,u_i) - k \le \sprod{x}{u_i} \le h(P,u_i) \quad \forall \ i=1,\ldots,d}
	\]
	one has $\maxfw(P') = \maxfw(P)$ and $P \subseteq P'$. Since $u_1,\ldots,u_d$ are linearly independent there exists $a \in \rational^d$ such that $\sprod{a}{u_i} = k-h(P,u_i)$ for every $i \in \{1,\ldots,d\}$. Then
	\[
		a + P' := \setcond{x \in \real^d}{ 0 \le \sprod{x}{u_i} \le k \quad \forall \ i=1,\ldots,d}.
	\]
	Let $B$ be the integer matrix with rows $u_1^\top,\ldots,u_d^\top$ (in this order). Then $a + P'$ can be given by
	\[
		a + P'= \setcond{x \in \real^d}{ B x \in [0,k]^d} = B^{-1}([0,k]^d).
	\]
	Thus 
	\[
		\vol(P) \le \vol(P') = \frac{1}{|\det B|} k^d \le k^d.
	\]
\end{proof}

\begin{proof}[Proof of Theorem~\ref{max-fac-width-of-max-lat-free}]
	The implication (ii) $\Rightarrow$ (i) follows from the invariance of the max-facet-width with respect to transformations from $\Aff(\integer^d)$. Let us show (i) $\Rightarrow$ (ii). We assume (i).
	We represent $\cL$ as the union $\bigcup_{i=0}^{d-1} \setcond{L \in \cL}{\dim\bigl(\rec(L)\bigr)=i}$. Thus, it suffices to show the finiteness of $\setcond{L \in \cL}{\dim \bigl(\rec(L))\bigr) = i} / \Aff(\integer^d)$ for each $i \in \{0,\ldots,d-1\}$.  In view of Proposition~\ref{max-lat-free-prop}.\ref{unbounded-max-lat-free} it suffices to consider the case $i=0$ only. That is, we restrict considerations to the case that for every $L \in \cL$ the set $L$ is a $d$-dimensional polytope. Let $L \in \cL$. By Proposition~\ref{max-lat-free-prop}.\ref{max-lat-free-description} the relative interior of each facet of the $d$-dimensional polytope $L$ contains an integral point; thus, the convex hull of such integral points is $d$-dimensional. It follows that the set $\cP:= \setcond{\conv(L \cap \integer^d)}{L \in \cL}$ consists of $d$-dimensional integral polytopes. By (i) and Proposition~\ref{vol:versus:maxfw} the volume of the polytopes $L \in \cL^d$ is bounded from above by $k^d$. Hence also the volume of the polytopes $P \in \cP$ is bounded by $k^d$.  Hence, by Theorem~\ref{finiteness:and:volume:boundedness}, $\cP/ \Aff(\integer^d)$ is finite. We fix $n \in \natur$ and polytopes $P_1,\ldots,P_n$ such that \[\cP / \Aff(\integer^d) = \{P_1,\ldots,P_n\}/ \Aff(\integer^d).\] For each $i \in \{1,\ldots,n\}$ we choose $q_i \in \intr(P_i) \cap \rational^d$. Furthermore we fix $m \in \natur$ such that $m q_i \in \integer^d$ for every $i \in \{1,\ldots,n\}$. Consider an arbitrary $L \in \cL$. By construction, there exists $A \in \Aff(\integer^d)$ and $i \in \{1,\ldots,n\}$ such that for $L':= A(L)$ one has $\conv(L'\cap \integer^d) = P_i$. The inclusion $P_i \subseteq L'$ implies
\[
	(P_i - q_i)^\circ \supseteq (L'-q_i)^\circ,
\]
where (in view of \eqref{polar:over:hyp:rep} and \eqref{P:repr:U(P)}) the polytope $(L'- q_i)^\circ$ can be given by 
\begin{equation} \label{L':rel}
	(L'- q_i)^\circ = \conv(X)
\end{equation}
with
\[
	X:= \setcond{\frac{u}{h(L'-q_i,u)}} {u \in U(L')}.
\]
For every $u \in U(L')$ one has
\begin{align*}
	0 & < h(L'-q_i,u) = h(L',u) + \sprod{-q_i}{u} \\ & \le h(L',u) + h(-L',u) = h(L'-L',u) = w(L',u) \le \maxfw(L') \le k.
\end{align*}
Hence $m (h(L'-q_i,u) )$ is a natural number not larger than $k m$. It follows that $(k \cdot m) !$ is divisible by $m (h(L',u) - \sprod{q_i}{u})$ and we deduce 
\[
	\frac{u}{h(L'-q_i,u)} \in \frac{1}{(k \cdot m)!} \integer^d.
\]
The latter implies $X \subseteq \frac{1}{(k \cdot m)!} \integer^d$. From $L=A(L')$ and \eqref{L':rel} we derive $L = A^{-1}\bigl(q_i+ \conv(X)^\circ\bigr)$. Thus, we have shown that for every $L \in \cL$ there exists $i \in \{1,\ldots,n\}$, a set $X \subseteq (P_i-q_i)^\circ \cap \left(\frac{1}{(k \cdot m) !} \integer^d \right)$ and a transformation $A \in \Aff(\integer^d)$ such that $L = A^{-1}(q_i + \conv(X)^\circ)$. Since, for every $i \in \{1,\ldots,n\}$, the set $(P_i-q_i)^\circ \cap \left(\frac{1}{(k \cdot m)!} \integer^d \right)$ containing $X$ is finite, there exist only finitely many choices of sets $X$ as above. This yields (ii) and finishes the proof of (i)  $\Rightarrow$ (ii).
\end{proof}

\begin{lemma} \label{boundedness:lemma}
	Let $(K_i)_{i \in \natur}$ be a decreasing sequence of closed convex sets in $\real^d$ (that is, $K_{i+1} \subseteq K_i$ for every $i \in \natur$). Let $K:= \bigcap_{i=1}^{+\infty} K_i$ be bounded and nonempty. Then there exists $j \in \natur$ such that, for every $i \in \natur$ with $i \ge j$, the set $K_i$ is bounded.
\end{lemma}
\begin{proof}
	We assume the contrary, that is, $(K_i)_{i \in \natur}$ contains infinitely many unbounded sets. Since $(K_i)_{i \in \natur}$ is decreasing it follows that, for every $i \in \natur$, the set $K_i$ is unbounded. For each $i \in \natur$ we choose a unit vector $u_i$ in $\rec(K_i).$ There exists an infinite subsequence $(u_{n_i})_{i=1}^{+\infty}$ of the sequence $(u_i)_{i=1}^{+\infty}$ such that $u_{n_i}$ converges to some unit vector $u \in \real^d$, as $i \rightarrow +\infty$. We fix an arbitrary $p \in K$. Consider arbitrary $i \in\natur$ and $t \ge 0$. By construction, one has $p+ t u_{n_j} \in K_i$ for every $j \in \natur$ with $n_j \ge i$. Taking limit, as $j \rightarrow +\infty$, we obtain $p+ t u \in K_i$. Since $i$ is arbitrary we deduce $p+ t u \in K$. Since $t \ge 0$ is arbitrary, it follows that $u \in \rec(K)$. Thus $K$ is unbounded, a contradiction to the assumptions.
\end{proof}

\begin{lemma}\label{rec:MI:hull}
	Let $\bM$ be a mixed-integer space given by \eqref{MI:space} and let $P \subseteq \real^d$ be a rational polyhedron with $P \cap \bM \ne \emptyset$. Then $\conv(P \cap \bM)$ is a rational polyhedron and one has
	\begin{align*}
		\rec\bigl(\conv(P \cap \bM)\bigr)  = \rec(P).
	\end{align*}
\end{lemma}
\begin{proof}
	It is known that $\conv(P\cap \bM)$ is a rational polyhedron (see, for example, \cite[\S\,16.7]{MR874114}). The inclusion $\rec\bigl(\conv(P \cap \bM)) \subseteq \rec(K)$ is trivial. Let us show the reverse inclusion. The cone $\rec(P)$ is an integral polyhedron. We represent $\rec(P)$ by $\rec(P) = \cone \bigl( \{a_1,\ldots,a_k\}\bigr)$, where $k \in \natur$ and $a_1,\ldots,a_k \in \integer^d$. Since $\bM+\integer^d = \bM$, for all $t \in \natur$, $x \in K \cap \bM$ and $i \in \{1,\ldots,k\}$ one has $x + t a_i \in K \cap \bM$. Since $K \cap \bM \ne \emptyset$, the latter shows $a_i \in \rec \bigl(\conv(K \cap \bM)\bigr)$ for every $i \in \{1,\ldots,k\}$ and yields $\rec(K) \subseteq \rec \bigl(\conv(K \cap \bM)\bigr)$. This shows the assertion.
\end{proof}

\begin{proof}[Proof of Theorem~\ref{convergence:thm}]
 	Let $K$ be an arbitrary line-free closed convex set in $\real^d$ and let $R:= \bigcap_{i=0}^{+\infty} R_\cL^i(K)$. In view of Theorem~\ref{descr:R}.\ref{rep:R}, $R_\cL^i(K)$ is closed for every $i \ge 0$. Hence, also $R$ is closed. We show
	\begin{equation} \label{ext:R:subset:M}
		\ext(R) \subseteq M.
	\end{equation}
	Assume \eqref{ext:R:subset:M} is not fulfilled. Then one can find $p \in \ext(R)$ with $p \not\in M$. Since $p \not\in M$, we can choose $L \in \cL$ with $p \in \intr(L)$. We also choose a bounded open set $U$ satisfying $x \in U \subseteq \intr(L)$. By the cap lemma (Lemma~\ref{cap:lemma}) there exists a hyperplane $H$ such that $p$ and $R \setminus U$ lie in different open halfspaces defined by $H$. Let $H^-$ and $H^+$ be the closed halfspaces defined by $H$ and satisfying $R \setminus U \subseteq \intr(H^+)$ and $p \in \intr(H^-)$. By the choice of $H$, we have $R \cap H^- \subseteq U$ and thus, $R \cap H^-$ is bounded. Since $R \cap H^- = \bigcap_{i=0}^{+\infty} \bigl(R_\cL^i(K) \cap H^- \bigr)$, applying Lemma~\ref{boundedness:lemma} we obtain the existence of $i \in \natur$ such that for every $j \in \natur$ with $j \ge i$ the set $R_\cL^j(K) \cap H^-$ is bounded. By Lemma~\ref{decr:convergence}, the sequence $\bigl(R_\cL^j(K) \cap H^-\bigr)_{j=i}^{+\infty}$ converges to $R \cap H^-$ in the Hausdorff distance. We recall that for $\rho>0$ and $x \in \real^d$ by $B(x,\rho)$ we denote the closed ball of radius $\rho$ centered at $x$. We can choose $\rho>0$ such that 
	\begin{equation} \label{extension:of:cap:lies:in:U}
		(R \cap H^-)+ B(o,\rho) \subseteq U.
	\end{equation} 
	In view of the convergence of $R_\cL^j(K) \cap H^-$ to $R \cap H^-$ and by the definition of the Hausdorff distance, there exists $i \in \natur$ such that 
	\begin{equation} \label{seq:caps:contained}
		R_\cL^i(K) \cap H^- \subseteq ( R \cap H^-) + B(o,\rho).
	\end{equation}
	With $i$ as above, using \eqref{extension:of:cap:lies:in:U} and \eqref{seq:caps:contained} we derive
	\begin{align*}
		R_\cL^i(K) \setminus \intr(L) & \subseteq R_\cL^i(K) \setminus U = \Bigl(\bigl(R_\cL^i(K) \cap H^+\bigr) \setminus U \Bigr)\cup \Bigl(\bigl(R_\cL^i(K) \cap H^-\bigr) \setminus U\bigr) \\
		& \subseteq H^+ \cup \Bigl(\bigl(R_\cL^i(K) \cap H^-\bigr) \setminus U\bigr) \subseteq H^+ \cup \Bigl(\bigl((R \cap H^-) + B(o,\rho)\bigr) \setminus U\bigr) = H^+.
	\end{align*}
	Thus, $R_\cL^i(K) \setminus \intr(L) \subseteq H^+$. It follows
	\begin{align*}
		R \subseteq R_\cL^{i+1}(K) = R_\cL\bigl(R_\cL^i(K)\bigr) \subseteq R_L\bigl(R_\cL^i(K)\bigr) = \conv \left( R_\cL^i(K) \setminus \intr(L) \right) \subseteq \conv(H^+) = H^+.
	\end{align*}
	We derived the inclusion $R \subseteq H^+$ which contradicts $p \in R \cap \intr(H^-)$. This shows \eqref{ext:R:subset:M}. 

	Let us verify Part~I. We notice that if $K \cap M = \emptyset$, then Part~I is fulfilled, since in this case one has $R=\emptyset$. This can be shown by contradiction. Assume that $K \cap M = \emptyset$ and $R \ne \emptyset$. Then $\ext(R) \ne \emptyset$. By \eqref{ext:R:subset:M} we get $\ext(R) \subseteq M$. Thus, $K \cap M \ne \emptyset$, a contradiction.  Therefore for the rest of the proof we assume $K \cap M \ne \emptyset$.  For every integer $i \ge 0$ we have:
	\begin{equation}
		\text{$R_\cL^i(K)$ is a closed set and $\rec\bigl(R_\cL^i(K)\bigr) = \rec(K)$.} \label{all:closures:closed}
	\end{equation}
	Condition \eqref{all:closures:closed} can be verified for every integer $i \ge 0$ by induction on $i$. For $i=0$ the condition trivially holds. If the condition holds for some integer $i \ge 0$, then (since $\rec(L)$ is a linear space) we can apply Theorem~\ref{descr:R}.\ref{rep:R}. It follows that $R_L\bigl(R_\cL^i(K)\bigr)$ is closed and one has $\rec\bigl(R_L(R_\cL^i(K))\bigr) = \rec\bigl(R_\cL^i(K)\bigr) = \rec(K)$ for every $L \in \cL$. Consequently $R_\cL^{i+1}(K)$ is closed and one has $\rec(R_\cL^{i+1}(K))=\rec(K)$. This yields \eqref{all:closures:closed} for every integer $i \ge 0$. From \eqref{all:closures:closed} and the definition of $R$ we deduce that $R$ is closed and one has $\rec(R) = \rec(K)$. Using the above observations and the inclusion $K \cap M \subseteq R$ we derive $\conv(K \cap M) + \rec(K) \subseteq R$. The reverse inclusion is obtained by representing $R$ with the help of \eqref{mink:sum:rep} and applying \eqref{ext:R:subset:M} together with $\rec(R)=\rec(K)$:
	\[
	  R = \conv\bigl(\ext(R)\bigr) + \rec(R) \subseteq \conv\bigl(K \cap M\bigr) + \rec(R) = \conv\bigl(K \cap M\bigr) + \rec(K).
	\]
	This shows Part~I. Let us show Part~II. Assume that $M$ is a mixed-integer space and let $P$ be a rational line-free polyhedron in $\real^d$. The case $P \cap M= \emptyset$ follows directly by Part~I. Therefore we assume $P \cap M \ne \emptyset$. By Lemma~\ref{rec:MI:hull} and Part~I we obtain $R = \conv(K \cap M) + \rec(K) = \conv(K \cap M) +\rec(\conv(K \cap M)) = \conv(K \cap M)$. This finishes the proof of Part~II.
\end{proof}

\begin{proof}[Proof of Corollary~\ref{convergence:bounded:case}]
	By Theorem~\ref{convergence:thm}.II one has $\bigcap_{i=0}^{+\infty} R_\cL^i(K) = \conv(K \cap M)$. Since the sequence $\bigl(R_\cL^i(K)\bigr)_{i=0}^{+\infty}$ consists of bounded closed convex sets and is decreasing, applying Lemma~\ref{decr:convergence}, we see that $R_\cL^i(K)$ converges to $\conv(K\cap M)$, as $i \rightarrow +\infty$, with respect to the Hausdorff distance.
\end{proof}

\small 


\begin{thebibliography}{MMWW02}

\bibitem[ALW10]{MR2676765}
K.~Andersen, Q.~Louveaux, and R.~Weismantel, \emph{An analysis of mixed integer
  linear sets based on lattice point free convex sets}, Math. Oper. Res.
  \textbf{35} (2010), no.~1, 233--256. \MR{2676765}

\bibitem[ALWW07]{MR2480507}
K.~Andersen, Q.~Louveaux, R.~Weismantel, and L.~A. Wolsey, \emph{Inequalities
  from two rows of a simplex tableau}, Integer programming and combinatorial
  optimization, Lecture Notes in Comput. Sci., vol. 4513, Springer, Berlin,
  2007, pp.~1--15. \MR{2011b:90082}

\bibitem[Ave11]{arXiv:1103.0629}
G.~Averkov, \emph{On the size of lattice simplices with a single interior
  lattice point}, Preprint arXiv:1103.0629, 2011.

\bibitem[AW10]{arXiv:1003.4365}
G.~Averkov and Ch. Wagner, \emph{Inequalities for the lattice width of
  lattice-free convex sets in the plane}, Preprint arXiv:1003.4365, 2010.

\bibitem[AWW10]{arXiv:1010.1077}
G.~Averkov, Ch. Wagner, and R.~Weismantel, \emph{Maximal lattice-free
  polyhedra: finiteness and an explicit description in dimension three},
  Preprint arXiv:1010.1077, 2010.

\bibitem[Bal71]{MR0290793}
E.~Balas, \emph{Intersection cuts---a new type of cutting planes for integer
  programming}, Operations Res. \textbf{19} (1971), 19--39. \MR{44 \#7972}

\bibitem[Bal74]{MR0376142}
\bysame, \emph{Disjunctive programming: cutting planes from logical
  conditions}, Nonlinear programming, 2 ({P}roc. {S}ympos. {S}pecial {I}nterest
  {G}roup on {M}ath. {P}rogramming, {U}niv. {W}isconsin, {M}adison, {W}is.,
  1974), Academic Press, New York, 1974, pp.~279--312. \MR{51 \#12328}

\bibitem[Bal85]{MR791175}
\bysame, \emph{Disjunctive programming and a hierarchy of relaxations for
  discrete optimization problems}, SIAM J. Algebraic Discrete Methods
  \textbf{6} (1985), no.~3, 466--486. \MR{86m:90101}

\bibitem[Bal98]{MR1663099}
\bysame, \emph{Disjunctive programming: properties of the convex hull of
  feasible points}, Discrete Appl. Math. \textbf{89} (1998), no.~1-3, 3--44.
  \MR{99k:90101}

\bibitem[BC09]{MR2555335}
V.~Borozan and G.~Cornu{\'e}jols, \emph{Minimal valid inequalities for integer
  constraints}, Math. Oper. Res. \textbf{34} (2009), no.~3, 538--546.
  \MR{2010i:90067}

\bibitem[BCM10]{BasuCornuejolsMargot10}
A.~Basu, G.~Cornu\'ejols, and F.~Margot, \emph{Intersection cuts with infinite
  split rank}, http://integer.tepper.cmu.edu/webpub/splrk.pdf, 2010.

\bibitem[BV92]{MR1191566}
I.~B{\'a}r{\'a}ny and A.~M. Vershik, \emph{On the number of convex lattice
  polytopes}, Geom. Funct. Anal. \textbf{2} (1992), no.~4, 381--393.
  \MR{93k:52013}

\bibitem[BW05]{BertWeiBook}
D.~Bertsimas and R.~Weismantel, \emph{Optimization over {I}ntegers, {D}ynamic
  {I}deas}, Belmont, MA, 2005.

\bibitem[Cad10]{MR2546330}
F.~Cadoux, \emph{Computing deep facet-defining disjunctive cuts for
  mixed-integer programming}, Math. Program. \textbf{122} (2010), no.~2, Ser.
  A, 197--223. \MR{2010f:90095}

\bibitem[Chv73]{MR0313080}
V.~Chv{\'a}tal, \emph{Edmonds polytopes and a hierarchy of combinatorial
  problems}, Discrete Math. \textbf{4} (1973), 305--337. \MR{47 \#1635}

\bibitem[CKS90]{MR1059391}
W.~Cook, R.~Kannan, and A.~Schrijver, \emph{Chv\'atal closures for mixed
  integer programming problems}, Math. Programming \textbf{47} (1990), no.~2,
  (Ser. A), 155--174. \MR{91c:90077}

\bibitem[CL06]{MR2216796}
G.~Cornu{\'e}jols and C.~Lemar{\'e}chal, \emph{A convex-analysis perspective on
  disjunctive cuts}, Math. Program. \textbf{106} (2006), no.~3, Ser. A,
  567--586. \MR{2007a:90103}

\bibitem[DPW10]{DelPiaWeismantel10}
A.~Del~Pia and R.~Weismantel, \emph{On convergence in mixed integer
  programming}, manuscript, 2010.

\bibitem[DW08]{MR2481733}
S.~S. Dey and L.~A. Wolsey, \emph{Lifting integer variables in minimal
  inequalities corresponding to lattice-free triangles}, Integer programming
  and combinatorial optimization, Lecture Notes in Comput. Sci., vol. 5035,
  Springer, Berlin, 2008, pp.~463--475. \MR{2011a:52014}

\bibitem[Esp10]{MR2593417}
D.~G. Espinoza, \emph{Computing with multi-row {G}omory cuts}, Oper. Res. Lett.
  \textbf{38} (2010), no.~2, 115--120. \MR{2010k:90203}

\bibitem[Gom58]{MR0102437}
R.~E. Gomory, \emph{Outline of an algorithm for integer solutions to linear
  programs}, Bull. Amer. Math. Soc. \textbf{64} (1958), 275--278. \MR{21
  \#1230}

\bibitem[Gom63]{MR0174390}
\bysame, \emph{An algorithm for integer solutions to linear programs}, Recent
  advances in mathematical programming, McGraw-Hill, New York, 1963,
  pp.~269--302. \MR{30 \#4594}

\bibitem[Gru07]{MR2335496}
P.~M. Gruber, \emph{Convex and {D}iscrete {G}eometry}, Grundlehren der
  Mathematischen Wissenschaften [Fundamental Principles of Mathematical
  Sciences], vol. 336, Springer, Berlin, 2007. \MR{2008f:52001}

\bibitem[Hen83]{MR688412}
D.~Hensley, \emph{Lattice vertex polytopes with interior lattice points},
  Pacific J. Math. \textbf{105} (1983), no.~1, 183--191. \MR{84c:52016}

\bibitem[Jer77]{MR0452674}
R.~G. Jeroslow, \emph{Cutting-plane theory: disjunctive methods}, Studies in
  integer programming ({P}roc. {W}orkshop, {B}onn, 1975), North-Holland,
  Amsterdam, 1977, pp.~293--330. Ann. of Discrete Math., Vol. 1. \MR{56
  \#10953}

\bibitem[J{\"o}r08]{JoergThesis2008}
M.~J{\"o}rg, \emph{$k$-disjunctive cuts and cutting plane algorithms for
  general mixed integer linear programs}, PhD-thesis, Technische
  Universit{\"a}t M{\"u}nchen, 2008.

\bibitem[Kru72]{MR0306057}
J.~B. Kruskal, \emph{The theory of well-quasi-ordering: {A} frequently
  discovered concept}, J. Combinatorial Theory Ser. A \textbf{13} (1972),
  297--305. \MR{46 \#5184}

\bibitem[Kur66]{MR0217751}
K.~Kuratowski, \emph{Topology. {V}ol. {I}}, New edition, revised and augmented.
  Translated from the French by J. Jaworowski, Academic Press, New York, 1966.
  \MR{36 \#840}

\bibitem[Lov89]{MR1114315}
L.~Lov{\'a}sz, \emph{Geometry of numbers and integer programming}, Mathematical
  programming ({T}okyo, 1988), Math. Appl. (Japanese Ser.), vol.~6, SCIPRESS,
  Tokyo, 1989, pp.~177--201. \MR{92f:90041}

\bibitem[LZ91]{MR1138580}
J.~C. Lagarias and G.~M. Ziegler, \emph{Bounds for lattice polytopes containing
  a fixed number of interior points in a sublattice}, Canad. J. Math.
  \textbf{43} (1991), no.~5, 1022--1035. \MR{92k:52032}

\bibitem[Mil85]{MR818505}
E.~C. Milner, \emph{Basic wqo- and bqo-theory}, Graphs and order ({B}anff,
  {A}lta., 1984), NATO Adv. Sci. Inst. Ser. C Math. Phys. Sci., vol. 147,
  Reidel, Dordrecht, 1985, pp.~487--502. \MR{87h:04004}

\bibitem[MMWW02]{MR1922341}
H.~Marchand, A.~Martin, R.~Weismantel, and L.~Wolsey, \emph{Cutting planes in
  integer and mixed integer programming}, Discrete Appl. Math. \textbf{123}
  (2002), no.~1-3, 397--446, Workshop on Discrete Optimization, DO'99
  (Piscataway, NJ). \MR{2003g:90045}

\bibitem[NW99]{MR1699321}
G.~Nemhauser and L.~Wolsey, \emph{Integer and {C}ombinatorial {O}ptimization},
  Wiley-Interscience Series in Discrete Mathematics and Optimization, John
  Wiley \& Sons Inc., New York, 1999, Reprint of the 1988 original, A
  Wiley-Interscience Publication. \MR{2000c:90001}

\bibitem[OM01]{MR1814548}
J.~H. Owen and S.~Mehrotra, \emph{A disjunctive cutting plane procedure for
  general mixed-integer linear programs}, Math. Program. \textbf{89} (2001),
  no.~3, Ser. A, 437--448. \MR{2001j:90051}

\bibitem[Pad05]{MR2176841}
M.~Padberg, \emph{Classical cuts for mixed-integer programming and
  branch-and-cut}, Ann. Oper. Res. \textbf{139} (2005), 321--352.
  \MR{2006e:90078}

\bibitem[Pik01]{MR1996360}
O.~Pikhurko, \emph{Lattice points in lattice polytopes}, Mathematika
  \textbf{48} (2001), no.~1-2, 15--24 (2003). \MR{2004f:52009}

\bibitem[Pou85]{MR818506}
M.~Pouzet, \emph{Applications of well quasi-ordering and better
  quasi-ordering}, Graphs and order ({B}anff, {A}lta., 1984), NATO Adv. Sci.
  Inst. Ser. C Math. Phys. Sci., vol. 147, Reidel, Dordrecht, 1985,
  pp.~503--519. \MR{87g:06014}

\bibitem[PS98]{MR1637890}
C.~H. Papadimitriou and K.~Steiglitz, \emph{Combinatorial {O}ptimization:
  {A}lgorithms and {C}omplexity}, Dover Publications Inc., Mineola, NY, 1998,
  Corrected reprint of the 1982 original. \MR{1637890}

\bibitem[Roc70]{MR0274683}
R.~T. Rockafellar, \emph{Convex {A}nalysis}, Princeton Mathematical Series, No.
  28, Princeton University Press, Princeton, N.J., 1970. \MR{43 \#445}

\bibitem[Sch86]{MR874114}
A.~Schrijver, \emph{Theory of {L}inear and {I}nteger {P}rogramming},
  Wiley-Interscience Series in Discrete Mathematics, John Wiley \& Sons Ltd.,
  Chichester, 1986, A Wiley-Interscience Publication. \MR{88m:90090}

\bibitem[Sch93]{MR1216521}
R.~Schneider, \emph{Convex {B}odies: {T}he {B}runn-{M}inkowski {T}heory},
  Encyclopedia of Mathematics and its Applications, vol.~44, Cambridge
  University Press, Cambridge, 1993. \MR{94d:52007}

\bibitem[ZPW82]{MR651251}
J.~Zaks, M.~A. Perles, and J.~M. Wilks, \emph{On lattice polytopes having
  interior lattice points}, Elem. Math. \textbf{37} (1982), no.~2, 44--46.
  \MR{83d:52012}

\end{thebibliography}

\providecommand{\bysame}{\leavevmode\hbox to3em{\hrulefill}\thinspace}
\providecommand{\MR}{\relax\ifhmode\unskip\space\fi MR }
\providecommand{\MRhref}[2]{%
  \href{http://www.ams.org/mathscinet-getitem?mr=#1}{#2}
}
\providecommand{\href}[2]{#2}

\end{document}